  \documentclass[a4paper,11pt]{amsproc} 
\newcommand{\optional}[1]{\relax}

\usepackage[centertags]{amsmath}
\usepackage{amssymb,amsfonts,amsthm}
\usepackage{lesch-macro} 

\nodetails
\nocomments%
\begin{document}

\copyrightinfo{~2010}{Matthias Lesch}

\title[$\Psi$DO and Regularized Traces]
{Pseudodifferential Operators and Regularized Traces} 

\author{Matthias Lesch}
\address{Mathematisches Institut,
Universit\"at Bonn, Beringstr. 6, D-53115 Bonn, Germany}
\email{ml@matthiaslesch.de, lesch@math.uni-bonn.de}
\urladdr{www.matthiaslesch.de, www.math.uni-bonn.de/people/lesch}

\thanks{The author gratefully acknowledges the hospitality of the
Department of Mathematics at The University of Colorado at Boulder
where this paper was written.
\\The author was partially supported by 
the Hausdorff Center for Mathematics (Bonn). }

\subjclass[2000]{Primary 58J42; Secondary 58J40, 58J35, 58B34}

\keywords{pseudodifferential operator, noncommutative residue, Dixmier
trace}


\begin{abstract}
This is a survey on trace constructions on various operator
algebras with an emphasis on regularized traces on algebras of
pseudodifferential operators. 
For motivation our point of departure is the classical Hilbert space
trace which is the unique semifinite \emph{normal} trace on the algebra
of bounded operators on a separable Hilbert space. Dropping the normality
assumption leads to the celebrated Dixmier traces.

Then we give a leisurely introduction to pseudodifferential operators. 
The parameter dependent calculus is emphasized and
it is shown how this calculus leads naturally to the asymptotic expansion of the resolvent trace
of an elliptic differential operator. 

The Hadamard partie finie regularization of an integral is explained and
used to extend the Hilbert space trace 
to the Kontsevich-Vishik canonical trace on pseudodifferential operators
of non--integral order.

Then the stage is well prepared for the residue trace of Wodzicki-Guillemin
and its purely functional analytic interpretation as a Dixmier trace
by Alain Connes.

We also discuss existence and uniqueness of traces for the algebra of parameter dependent 
pseudodifferential operators; the results are surprisingly different.

Finally, we will discuss the analogue of the regularized traces on the
symbolic level and study the de Rham cohomology of $\R^n$ with coefficients
being symbol functions. This generalizes a recent result of S.
Paycha concerning the characterization of the Hadamard partie finie
integral and the residue integral in light of the Stokes property. 
\end{abstract}
\maketitle
\tableofcontents
\listoffigures
\section{Introduction} 
Traces on an algebra are important linear functionals which come up in
various incarnations in various branches of mathematics, e.g. group
characters, norm and trace in field extensions, many trace formulas,
to mention just a few. 

On a separable Hilbert space $\cH$ there is a canonical trace (tracial
weight, see Section \plref{s:HST}) $\Tr$ defined on non--negative
operators by
\begin{equation}
    \Tr(T):=\sum_{j=0}^\infty \scalar{T e_j}{e_j},
\end{equation}
where $(e_j)_{j\ge 0}$ is an orthonormal basis.
This is the unique semifinite normal trace on the algebra $\cB(\cH)$ of
bounded operators on $\cH$.
In the 1930's \textsc{Murray} and \textsc{von Neumann} \cite{MurNeu:ROI}, \cite{MurNeu:ROII},
\cite{Neu:ROIII}, \cite{MurNeu:ROIV} studied traces on weakly closed
$*$--subalgebras (now known as von Neumann algebras) of $\cB(\cH)$.
They showed that on a von Neumann \emph{factor} there is up to a
normalization a unique semifinite normal trace.

\begin{sloppy}
\textsc{Guillemin} \cite{Gui:NPW} and \textsc{Wodzicki} 
\cite{Wod:LIS}, \cite{Wod:NRF} discovered independently
that a similar uniqueness statement holds for the algebra
of pseudodifferential operators on a compact manifold. The \emph{residue
trace}, however, has nothing to do with the Hilbert space trace: it vanishes
on trace class operators.
\end{sloppy}

In the 60s \textsc{Dixmier} \cite{Dix:ETN} had already proved that the uniqueness
statement for the Hilbert space trace fails if one gives up the assumption
that the trace is normal. 

In the late 80's and early 90's then the Dixmier trace had a 
celebrated comeback when \textsc{Alain Connes} \cite{Con:AFN}
proved that in important cases the residue trace coincides with a Dixmier trace.

The aim of this note is to survey some of these results. We will not touch
von Neumann algebras, however, any further.

The paper is organized as follows:

In Section \plref{s:HST} our point of departure is the classical Hilbert space
trace. We give a short proof that it is up to a factor
the unique normal tracial weight on the algebra $\cB(\cH)$
of bounded operators on a separable Hilbert space $\cH$.

Then we reproduce Dixmier's very elegant construction which
shows that non--normal tracial weights are abundant. We do
confine ourselves however to those Dixmier traces which will
later turn out to be related to the residue trace.

Section \plref{s:POP} presents the basic calculus of
pseudodifferential operators with parameter on a closed manifold.

In Section \plref{s:EHS} we pause the discussion of pseudodifferential operators
and look at the problem of extending the Hilbert space trace to
pseudodifferential operators of higher order. A pseudodifferential operator
$A$ of order $<-\dim M$ on a closed manifold $M$ is of trace class
and its trace is given by integrating is Schwartz kernel $k_A(x,y)$ over
the diagonal
\begin{equation}\label{intro-2}
\Tr(A)=\int_M k_A(x,x) dx.
\end{equation}
We will show that the classical Hadamard partie finie regularization of
integrals allows to extend Eq. \eqref{intro-2} to all pseudodifferential
operators of non--integral order. This is the celebrated Kontsevich-Vishik
canonical trace.

Section \plref{s:POPAE} on asymptotic analysis 
then shows how the parameter dependent pseudodifferential
calculus leads naturally to the asymptotic expansion of the resolvent trace
of an elliptic differential operator. For the resolvent of elliptic pseudodifferential operators a refinement, 
due to Grubb and Seeley, of the parametric calculus is necessary. Without going into the details of this refined calculus
we will explain why additional $\log \lambda$ terms appear in the asymptotic expansion of $\Tr(B (P-\gl)^{-N})$ if $B$ or
$P$ are pseudodifferential rather than differential operators. These $\log \lambda$ terms are at the heart of
the noncommutative residue trace.
The straightforward relations between the resolvent expansion, the heat
trace expansion and the meromorphic continuation of the $\zeta$--function, which are
based on the Mellin transform respectively a contour integral method, are
also briefly discussed.

In Section \plref{s:RT} we state the main result about the existence and
uniqueness of the residue trace. We present it in a slightly generalized form
due to the author for $\log$--polyhomogeneous pseudodifferential operators. 
A formula for the relation between the residue trace of a power of the
Laplacian and the Einstein--Hilbert action due to \textsc{Kalau--Walze} \cite{KalWal:GNC}
and \textsc{Kastler} \cite{Kas:DOG} is proved in an example.

Then we give a proof of Connes' Trace Theorem which states that on
pseudodifferential operators of order minus $\dim M$ on a closed manifold $M$
the residue trace is proportional to the Dixmier trace. 

Having seen the significance of the parameter dependent calculus 
it is natural to ask whether the algebras of parameter dependent
pseudodifferential operators have an analogue of the residue trace. 
Somewhat surprisingly the results for these algebras are quite different:
there are many traces on this algebra, however, there is a unique symbol--valued trace from which many other traces
can be derived. This result resembles very much the 
center valued trace in von Neumann algebra theory.
Furthermore, in contrast to the non--parametric case the
$L^2$--Hilbert space trace extends to a trace on the whole algebra. 
This part of the paper surveys results from a joint paper with
\textsc{Markus J. Pflaum} \cite{LesPfl:TAP}.

Finally, in the short Section \plref{s:DFC} we will discuss 
the analogue of the regularized traces on the
symbolic level and announce a generalization of a recent result of S.
Paycha concerning the characterization of the Hadamard partie finie
integral and the residue integral in light of the Stokes property. 
The result presented here allows one to calculate de Rham cohomology 
groups of forms on $\R^n$ whose coefficients lie in a certain symbol space.
We will show that both the Hadamard partie finie integral and the residue
integral provide an integration along the fiber on the cone 
$\R_+^*\times M$ and as a consequence there is an analogue of
the Thom isomorphism.

\textsc{Acknowledgments.} 
I would like to thank the organizers of the conference on Motives, Quantum
Field Theory and Pseudodifferential Operators for inviting me to contribute these notes.
Also I would like to thank the anonymous referee
for taking his job very seriously and for making very detailed remarks
on how to improve the paper. I think the paper has benefited considerably
from those remarks.
\section{The Hilbert space trace (tracial weight)}
\label{s:HST}

\subsection{Basic definitions}
Let $\cH$ be a separable Hilbert space. Denote by $\cB(\cH)$ the algebra
of bounded operators on $\cH$. Let $\cA$ be a $C^*$--subalgebra, that is,
a norm closed self--adjoint ($a\in\cA\Rightarrow a^*\in\cA$) subalgebra.
It follows that $\cA$ is invariant under continuous functional calculus,
e.g. if $a\in\cA$ is non--negative then $\sqrt{a}\in\cA$.

Denote by $\cA_+\subset \cA$ the set of non--negative elements.
$\cA_+$ is a cone in the following sense:
\begin{enumerate}
\item $T\in\cA_+, \gl\in\R_+ \Rightarrow \gl T\in\cA_+,$
\item $S,T\in\cA_+, \gl,\mu\in\R_+\Rightarrow \gl S+\mu T\in\cA_+$.
\end{enumerate}

A \emph{weight} on $\cA$ is a map
\begin{equation}
\tau: \cA_+\longrightarrow \R_+\cup \{\infty\},\quad \R_+:=[0,\infty),
\end{equation}
such that 
\begin{equation}\label{eq:def-weight}
\tau (\gl S+\mu T)= \gl \tau(S)+\mu \tau(T),\quad \gl,\mu\ge 0,\;
S,T\in\cA_+.
\end{equation}
A weight is called \emph{tracial} if
\begin{equation}\label{eq:def-tracial-weight}
  \tau(TT^*)=\tau(T^* T), \quad T\in\cA_+.
\end{equation}

It follows from \eqref{eq:def-tracial-weight} that for a unitary $U\in\cA$ and $T\in\cA_+$
\begin{equation}
\begin{split}
    \tau(UTU^*)=
    \tau((UT^{1/2})(UT^{1/2})^*)=\tau((UT^{1/2})^*(UT^{1/2}))=\tau(T). 
\end{split}
\end{equation}
\eqref{eq:def-weight} implies that $\tau$ is monotone in the sense that if $0\le S\le T$ then 
\begin{equation}
  \tau(T)=\tau(S)+\tau(T-S)\ge \tau(S).
\end{equation}

\begin{remark}
In the literature tracial weights are often just called traces.
We adopt here the convention of \textsc{Kadison} and \textsc{Ringrose}
\cite[Chap. 8]{KadRin:FTOII}. 

We reserve the word trace for a linear functional $\tau:\cR\longrightarrow \C$ 
on a $\C$--algebra $\cR$ which satisfies $\tau(AB)=\tau(BA)$ for 
$A,B\in\cR$.
A priori a tracial weight $\tau$ is only defined on the positive cone
of $\cA$ and it may take the value $\infty$. Below we will see that
there is a natural ideal in $\cA$ on which $\tau$ is a trace.
\end{remark}

\subsubsection{The canonical tracial weight on bounded operators on a
Hilbert space}
Let $(e_j)_{j\in \Z_+}$ be an orthonormal basis of the Hilbert space $\cH$;
$\Z_+:=\{0,1,2,\ldots\}$. For $T\in\cB_+(\cH)$
put
\begin{equation}\label{eq:def-tr}
    \Tr(T):=\sum_{j=0}^\infty \scalar{T e_j}{e_j}.
\end{equation}
$\Tr(T)$ is indeed independent of the choice of the orthonormal basis
and it is a tracial weight on $\cB(\cH)$ (\textsc{Pedersen} \cite[Sec.
3.4]{Ped:AN}).

\subsubsection{Trace ideals}\label{ss:trace-ideals} We return to the general set--up of 
a tracial weight on a $C^*$--subalgebra $\cA\subset\cB(\cH)$. 
Put
\begin{equation}\label{eq:trace-ideal-1}
\cL_+^1(\cA,\tau):=\bigsetdef{T\in\cA_+}{ \tau(T)<\infty}
\end{equation}
and denote by $\cL^1(\cA,\tau)$ the linear span of $\cL^1_+(\cA,\tau)$.
Furthermore, let
\begin{equation}\label{eq:trace-ideal-2}
\cL^2(\cA,\tau):=\bigsetdef{T\in\cA}{\tau(T^*T)<\infty}.
\end{equation}

Using the inequality
\begin{equation}\label{eq:20090508-3}
\begin{split}
  (S+T)^*(S+T)&\le (S+T)^*(S+T)+(S-T)^*(S-T)\\
              &=2 (S^*S+T^*T)
\end{split}
\end{equation}
and the polarization identity
\begin{equation}\label{eq:polarization}
   4 T^* S=\sum_{k=0}^3 i^k (S+i^k T)^*(S+ i^k T)
\end{equation}
one proves exactly as for the tracial weight $\Tr$ in \cite[Sec. 3.4]{Ped:AN}:

\begin{prop}\label{p:20090511-1} $\cL^1(\cA,\tau)$ and $\cL^2(\cA,\tau)$ are two--sided
self--adjoint ideals in $\cA$. 

Moreover for $T,S\in \cL^2(\cA,\tau)$ one has $TS,ST\in\cL^1(\cA,\tau)$
and 
\[
\tau(ST)=\tau(TS).\]
The same formula holds for $T\in\cL^1(\cA,\tau)$ and $S\in\cB(\cH)$.

In particular $\tau\restriction \cL^p(\cA,\tau), p=1,2,$ is a trace.
\end{prop}

\subsection{Uniqueness of $\Tr$ on $\cB(\cH)$}

As for finite--dimensional matrix algebras one now shows
that up to a normalization there is a unique trace on the ideal
of finite rank operators.

\begin{lemma}\label{l:Tr-uniqueness-FR} Let $\FRH$ be the ideal of finite rank operators on $\cH$.
Any trace $\tau:\FRH\longrightarrow \C$ is proportional to
$\Tr\restriction\FRH$.
\end{lemma}
\begin{proof}
Let $P,Q\in\cB(\cH)$ be rank one orthogonal projections. 
Choose $v\in\im P, w\in \im Q$ with $\|v\|=\|w\|=1$ and put
\begin{equation}
     T:= \scalar{v}{\cdot}\; w.
\end{equation}
Then $T\in\FRH$ and $T^*T=P, TT^*=Q$. Consequently $\tau$
takes the same value $\gl_\tau\ge 0$ on all orthogonal projections of
rank one.

If $T\in\FRH$ is self--adjoint then $T=\sum_{j=1}^N \mu_j P_j$
with rank one orthogonal projections $P_j$. Thus
\begin{equation}
    \tau(T)=\gl_\tau \sum_{j=1}^N\mu_j=\gl_\tau \Tr(T).
\end{equation}
Since each $T\in\FRH$ is a linear combination of self--adjoint
elements of $\FRH$ we reach the conclusion.
\end{proof}

The properties of $\Tr$ we have mentioned so far are not sufficient
to show that a tracial weight on $\cB(\cH)$ is proportional to $\Tr$.
The property which implies this is \emph{normality}:

\begin{prop}\label{p:tr-normal}
\textup{1.} $\Tr$ is \emph{normal}, that is, if $(T_n)_{n\in\Z_+}\subset\cB_+(\cH)$ is
an increasing sequence with $T_n\to T\in\cB_+(\cH)$ strongly then
$\Tr(T)=\sup_{n\in\Z_+} \Tr(T_n)$.

\textup{2.} Let $\tau$ be a normal tracial weight on $\cB(\cH)$. Then there
is a constant $\gl_\tau\in \R_+\cup \{\infty\}$ such that for $T\in\cB_+(\cH)$
we have $\tau(T)=\gl_\tau \Tr(T)$.
\end{prop}
\begin{remark} In the somewhat pathological case $\gl=\infty$ the tracial
weight $\tau_\infty$ is given by
\begin{equation}
    \tau_\infty(T)=\begin{cases} \infty,& T\in\cB_+(\cH)\setminus \{0\},\\
                                 0,& T=0.
				 \end{cases}
\end{equation}
In all other cases $\tau$ is \emph{semifinite}, that means for
$T\in\cB_+(\cH)$ there is an increasing sequence $(T_n)_{n\in \Z_+}$
with $\tau(T_n)<\infty$ and $T_n\nearrow T$ strongly. 
Here, $T_n$ may be chosen of finite rank.
\end{remark}
\begin{proof}
1. Let $(e_k)_{k\in\Z_+}$ be an orthonormal basis of $\cH$. Since $T_n \to
T$ strongly we have $\scalar{T_n e_k}{e_k}\nearrow \scalar{T e_k}{e_k}$.
The Monotone Convergence Theorem for the counting measure on $\Z_+$ then
implies
\begin{equation}
  \Tr(T)=\sum_{k=0}^ \infty \scalar{Te_k}{e_k}=\sup_{n\in\Z_+} \sum_{k=0}^\infty
\scalar{T_ne_k}{e_k}=\sup_{n\in\Z_+} \Tr(T_n).
\end{equation}

2. Let $\tau:\cB_+(\cH)\longrightarrow \R_+\cup \{\infty\}$ be a normal
tracial weight. As in the proof of Lemma \ref{l:Tr-uniqueness-FR} one shows
that $\tau\restriction \FRH=\gl_\tau \Tr\restriction\FRH$ for some
$\gl_\tau\in
\R_+\cup \{\infty\}$.

Choose an increasing sequence of orthogonal projections $(P_n)_{n\in
\Z_+}$, $\rank P_n=n$. Given $T\in\cB_+(\cH)$ the sequence of finite rank operators
$(T^{1/2}P_nT^{1/2})_{n\in\Z_+}$ is increasing and it converges strongly to
$T$. Since $\tau$ is assumed to be normal we thus find
\begin{equation}\begin{split}
\tau(T)&=\sup_{n\in\Z_+}  \tau(T^{1/2}P_nT^{1/2})\\
       &=\sup_{n\in\Z_+}  \gl_\tau \Tr(T^{1/2}P_nT^{1/2})=\gl_\tau\Tr(T).\qedhere
\end{split}
\end{equation}
\end{proof}

\begin{remark}
The uniqueness of the trace $\Tr$ we presented here is in fact a special
case of a rich theory of traces for weakly closed self--adjoint
subalgebras of $\cB(\cH)$ (von Neumann algebras) due to \textsc{Murray}
and \textsc{von Neumann} \cite{MurNeu:ROI}, \cite{MurNeu:ROII},
\cite{Neu:ROIII}, \cite{MurNeu:ROIV}.
\end{remark}

\subsection{The Dixmier Trace}
\label{ss:Dixmier-trace}

In view of Proposition \plref{p:tr-normal} it is natural to ask whether
there exist non--normal tracial weights on $\cB(\cH)$.  
A cheap answer to this question would be to define for $T\in\cB_+(\cH)$
\begin{equation}
   \tau(T):=\begin{cases} \Tr(T),& T\in\FRH,\\
                           \infty, & T\not\in\FRH.
\end{cases}
\end{equation}
Then $\tau$ is certainly a non--trivial non--normal tracial weight on
$\cB(\cH)$. 

To make the problem non--trivial, one should ask whether there exists a
non--trivial non--normal tracial weight on $\cB(\cH)$ which vanishes
on trace class operators.
This was answered affirmatively by \textsc{J. Dixmier} in the short note \cite{Dix:ETN}.
We briefly describe Dixmier's very elegant argument.

Denote by $\cK(\cH)$ the ideal
of compact operators. We abbreviate
\begin{equation}
\cL^p(\cH):=\cL^p(\cB(\cH),\Tr),
\end{equation}
see Section \plref{ss:trace-ideals}. 
A compact operator $T$ is in $\cL^1(\cH)$
if and only if $\sum\limits_{j=1}^\infty \mu_j(T)<\infty$. Here $\mu_j(T), j\ge 1,$ 
denotes the sequence of eigenvalues of $|T|$ counted with multiplicity.

By $\cL^{(1,\infty)}(\cH)\supset\cL^1(\cH)$ one denotes the space of 
$T\in\cK(\cH)$ for which 
\[\sum\limits_{j=1}^{N}\mu_j(T)=O(\log N),\quad  N\to\infty.\]

For an operator $T\in\cL^{(1,\infty)}(\cH)$ the sequence 
\[
\ga_N(T):=\frac{1}{\log (N+1)}\sum_{j=1}^{N} \mu_j(T),\quad  N\ge 1,\]
is thus bounded.
\begin{prop}[\textsc{J. Dixmier} {\cite{Dix:ETN}}]\label{p:Dixmier-Connes}
Let $\go\in l^\infty(\Z_+\setminus\{0\})^*$ be 
a  linear functional satisfying
\begin{enumerate}
\item[\textup{(1)}] $\go$ is a \emph{state}, that is, a 
positive linear functional with\\ $\go(1,1,\dots)=1$.
\item[\textup{(2)}] $\go((\ga_N)_{N\ge 1})=0$ if $\lim\limits_{N\to\infty} \ga_N=0$.
\item[\textup{(3)}]
  \begin{equation}\label{eq:20081114-3}
\go(\ga_1,\ga_2,\ga_3,\dots)=\go(\ga_1,\ga_1,\ga_2,\ga_2,\dots).
\end{equation}
\end{enumerate}
Put for non--negative $T\in\cL^{(1,\infty)}(\cH)$
\begin{equation}
  \begin{split}
   \Tr_\go(T)&:=\go\Bigl(\bigl(\frac{1}{\log (N+1)}\sum_{j=1}^{N}\mu_j(T)\bigr)_{N\ge 1}\Bigr)\\
              &=:\lim_\go \frac{1}{\log (N+1)}\sum_{j=1}^{N}\mu_j(T).
  \end{split}		
\end{equation}
Then $\Tr_\go$ extends by linearity to a trace on $\cL^{(1,\infty)}(\cH)$.
If $T\in\cL^1(\cH)$ is of trace class then $\Tr_\go(T)=0$ . Furthermore,
\begin{equation}
   \Tr_\go(T)=\lim_{N\to\infty}\frac{1}{\log (N+1)}\sum_{j=1}^{N}\mu_j(T),
\end{equation}
if the limit on the right hand side exists.

Finally, by putting $\Tr_\go(T)=\infty$
if $T\in\cB_+(\cH)\setminus \cL^{(1,\infty)}(\cH)$ 
one extends $\Tr_\go$ to $\cB_+(\cH)$ and hence one
obtains a non--normal
tracial weight on $\cB(\cH)$.
\end{prop}
\begin{proof}
Let us make a few comments on how this result is proved:
First the existence of a state $\go$ with the properties (1), (2), and (3)
can be shown by a fixed point argument; in this simple case even Schauder's
Fixed Point Theorem would suffice. Alternatively, the theory of Ces{\`a}ro
means leads to a more constructive proof of the existence of $\go$,
\textsc{Connes} \cite[Sec. 4.2.$\gamma$]{Con:NG}.

Next we note that (1) and (2) imply 
that if $(\ga_N)_{N\ge 1}$ is convergent then 
$\go((\ga_N)_{N\ge 1})=\lim\limits_{N\to\infty} \ga_N$. 
Thus changing finitely many terms of $(\ga_N)_{N\ge 1}$ 
(i.e. adding a sequence of limit $0$) does not change
its $\go$--limit. Together with the positivity of
$\go$ this implies
\begin{equation}\label{eq:20090511-2}
\text{if $\ga_N\le\gb_N$ for $N\ge N_0$ then $\go((\ga_N)_{N\ge 1})
\le \go((\gb_N)_{N\ge 1})$.}
\end{equation}


The previously mentioned facts imply furthermore
\begin{equation}\label{eq:20081114-2}
       \liminf_{N\to\infty}\ga_N\le \go((\ga_N)_{N\ge 1})\le \limsup_{N\to\infty}\ga_N.
\end{equation}

Now let $T_1,T_2\in\cL^{(1,\infty)}$ be non--negative operators and
put
\begin{equation}\begin{split}
      \ga_N&:=\frac{1}{\log (N+1)}\sum_{j=1}^{N}\mu_j(T_1),\quad 
              \gb_N:=\frac{1}{\log (N+1)}\sum_{j=1}^{N}\mu_j(T_2),\\
      \gamma_N&:=\frac{1}{\log (N+1)}\sum_{j=1}^{N}\mu_j(T_1+T_2).
  \end{split}
\end{equation}
Using the min-max principle one shows the inequalities
\begin{equation}\label{eq:maxmin-inequalities}
  \sum_{j=1}^N \mu_j(T_1+T_2)\le \sum_{j=1}^N \mu_j(T_1)+\mu_j(T_2)\le
  \sum_{j=1}^{2N} \mu_j(T_1+T_2),
\end{equation}
cf.  \textsc{Hersch} \cite{Her:CVS, Her:IVP}, thus 
\begin{align}
       \gamma_N&\le \ga_N+\gb_N,\label{eq:20090511-1}\\
       \ga_N+\gb_N &\le \frac{\log (2N+1)}{\log (N+1)}
       \gamma_{2N}.\label{eq:20081118-7}
\end{align}
\eqref{eq:20090511-1} gives $\go((\gamma_N)_{N\ge 1})\le
\go((\ga_N)_{N\ge 1})+\go((\gb_N)_{N\ge 1})$.

The proof of the converse inequality makes essential use
of the crucial assumption \eqref{eq:20081114-3}. Together with
\eqref{eq:20081118-7} and \eqref{eq:20090511-2} we find
\begin{equation}
\begin{split}
     \go((\ga_N)_{N\ge 1})+\go((\gb_N)_{N\ge 1}) & \le
\go(\gamma_2,\gamma_4,\gamma_6,\dots)\\
       &=\go(\gamma_2,\gamma_2,\gamma_4,\gamma_4,\dots),
\end{split}
\end{equation}
so, in view of \ref{p:Dixmier-Connes} (2),  it only remains to remark 
that \[\lim\limits_{N\to\infty} (\gamma_{2N}-\gamma_{2N-1})=0.\]

%

Thus $\Tr_\go$ is additive on the cone of positive operators. Since $\Tr_\go(T)$
depends only on the spectrum, it is certainly invariant under conjugation
by unitary operators. Now it is easy to see that $\Tr_\go$ extends by linearity
to a trace on $\cL^{(1,\infty)}(\cH)$. The other properties follow easily.
\end{proof}

\section{Pseudodifferential operators with parameter}
\label{s:POP}

\subsection{From differential operators to pseudodifferential operators}
Historically, pseudodifferential operators were invented to understand differential
operators. Suppose given a differential operator
\begin{equation}\label{eq:2-1}
    P=\sum_{|\ga|\le d} p_\ga(x)\; i^{-|\ga|} \frac{\pl^\ga}{\pl x^\ga}
\end{equation}
in an open set $U\subset\R^n$. Representing a function $u\in\cinfz{U}$
in terms of its Fourier transform
\begin{equation}
     u(x)=\int_{\R^n} e^{i \scalar{x}{\xi}} \hat u(\xi)\dbar\xi,\quad \dbar\xi=(2\pi)^{-n} d\xi,
     \label{eq:Fourier2-2}
\end{equation}
where $\hat u(\xi)=\int_{\R^n} e^{-i \scalar{x}{\xi}} u(x) dx$,
we find
\begin{equation}\label{eq:2.3}
   \begin{split}
      Pu(x)&= \int_{\R^n} e^{-i \scalar{x}{\xi}}p(x,\xi) \hat u(\xi)\dbar\xi \\
           &=\int_{\R^n}\Bigl(\int_U e^{i\scalar{x-y}{\xi}} p(x,\xi) u(y) dy \Bigr)\dbar\xi\\
           &=:\bigl(\Op(p) u\bigr)(x).
   \end{split}
\end{equation}
Here
\begin{equation}\label{eq:2.4}
   p(x,\xi)= \sum_{|\ga|\le d}p_\ga(x)\xi^\ga
\end{equation}
denotes the \emph{complete symbol} of $P$. The right hand side of \eqref{eq:2.3} shows
that $P$ is a pseudodifferential operator with complete symbol function
$p(x,\xi)$.

Note that $p(x,\xi)$ is a polynomial in $\xi$. One now considers 
pseudodifferential operators with more general symbol functions such that
inverses of differential operators are included into the calculus. E.g.
a first approximation to the resolvent $(P-\gl^d)\ii$ is given
by $\Op((p(\cdot,\cdot)-\gl^d)\ii)$. For constant coefficient differential
operators this is indeed the exact resolvent.

Let us now describe the most commonly used symbol spaces. In view of
the resolvent example above we are going to consider symbols with
an auxiliary parameter.

\pagebreak[3]
\subsection{Basic calculus with parameter}
We first recall the notion of conic manifolds and conic sets
from \textsc{Duistermaat} \cite[Sec. 2]{Dui:FIO}.
A conic manifold is a smooth principal fiber bundle $\Gamma \rightarrow B$
with structure group $\R_+^*:=(0,\infty)$. It is always trivializable. 
A subset $\Gamma \subset \R^N\setminus \{ 0 \} $ 
which is a conic manifold by the natural $\R_+^*$-action on $\R^N\setminus
\{0\}$ is called a conic set. 
The base manifold of a conic set $\Gamma \subset \R^N\setminus\{0\}$ is
diffeomorphic to $S \Gamma := \Gamma \cap S^{N -1}$. By a cone 
$\Gamma \subset \R^N$ we will always mean a conic set or the 
closure of a conic set in $\R^N$ such that $\Gamma$ has nonempty interior.
Thus $\R^N$ and $\R^N\setminus\{0\}$ are cones, but only the latter is a conic set. 
$\{0\}$ is a zero--dimensional cone.

\subsubsection{Symbols}
Let $U\subset \R^n$ be an open subset and $\Gamma\subset \R^N$ a cone. A typical
example we have in mind is $\Gamma=\R^n\times\Lambda$, where $\Lambda\subset\C$
is an open cone. 

We denote by $\sym^m(U;\Gamma)$, $m\in \R$, the space of symbols 
of H\"ormander type $(1,0)$ (\textsc{H\"ormander} \cite{Hor:FIOI}, 
\textsc{Grigis--Sj{\o}strand} \cite{GriSjo:MAD}). 
More precisely, $\sym^m(U;\Gamma)$ consists of those 
$a\in \CC^\infty(U\times \Gamma)$ such that for multi--indices 
$\alpha\in \Z_+^n,\gamma\in \Z_+^N$ and compact subsets $K\subset U, L\subset\Gamma$ 
we have an estimate
\begin{equation}\label{eq:3.1}
    \bigl|\partial_x^\alpha\partial_\xi^\gamma a(x,\xi)\bigr|
  \le C_{\alpha,\gamma,K,L} (1+|\xi|)^{m-|\gamma|}, \quad x\in K, \xi\in L^c.
\end{equation}
Here $L^c=\bigsetdef{t\xi}{\xi\in L, t\ge 1}$.
The best constants in \eqref{eq:3.1} provide a set of 
semi-norms which endow
$\sym^\infty (U;\Gamma):=\bigcup_{m\in\C}\sym^m(U;\Gamma)$ with the structure of a
Fr{\'e}chet algebra. 
We mention the following variants of the space $\sym^\bullet$:
\subsubsection{Classical symbols $\CS^m(U;\Gamma)$}
A symbol $a\in\sym^m(U;\Gamma)$ is called \emph{classical} if there are
$a_{m-j}\in \cinf{U\times\Gamma}$ with
\begin{equation}\label{eq:classical}
   a_{m-j}(x,r\xi)=r^{m-j} a_{m-j}(x,\xi),\quad r\ge 1, |\xi|\ge 1,
\end{equation}
such that for $N\in\Z_+$
\begin{equation}\label{eq:classical-a}
  a-\sum_{j=0}^{N-1} a_{m-j}\in\sym^{m-N}(U;\Gamma).
\end{equation}
The latter property is usually abbreviated $a\sim\sum\limits_{j=0}^\infty a_{m-j}$.

Many authors require the functions in \eqref{eq:classical} to be 
homogeneous everywhere on $\Gamma\setminus\{0\}$. Note however 
that if $\Gamma=\R^p$ and $f:\Gamma\to\C$ is a function which is homogeneous
of degree $\ga$ then $f$ cannot be smooth at $0$ unless $\ga\in\Z_+$. So
such a function is not a symbol in the strict sense. We prefer the
functions in the expansion \eqref{eq:classical-a} to be smooth everywhere
and homogeneous only for $r\ge 1$ and $|\xi|\ge 1$.

The space of classical symbols of order $m$ is denoted by $\CS^m(U;\Gamma)$.
In view of the asymptotic expansion \eqref{eq:classical-a} we have
$\CS^{m'}(U;\Gamma)\subset \CS^m(U;\Gamma)$ only if $m-m'\in\Z_+$ is a non--negative
integer. 
\subsubsection{$\log$--polyhomogeneous symbols $\CS^{m,k}(U;\Gamma)$}
$a\in \sym^m(U;\Gamma)$ is called
\emph{$\log$--polyhomogeneous} (cf.~\textsc{Lesch} \cite{Les:NRP}) of order
$(m,k)$ if it has an 
asymptotic expansion in $\sym^\infty (U;\Gamma)$ of the form
\begin{equation}\label{ML-G2.2}
    a\sim\sum\limits_{j=0}^\infty a_{m-j} \quad   
    \text{ with } a_{m-j}=\sum_{l=0}^{k} b_{m-j,l}, 
   \end{equation}
where $a_{m-j}\in \CC^\infty(U\times \Gamma)$ and 
$b_{m-j,l}(x,\xi)=\tilde b_{m-j,l}(x,\xi/|\xi|)|\xi|^{m-j}\log^l|\xi|$ for
$|\xi|\ge 1$. 

By $\CS^{m,k}(U;\Gamma)$ we denote 
the space of $\log$--polyhomogeneous symbols of order $(m,k)$.
Classical symbols are those of $\log$ degree $0$, i.e.
$\CS^m(U;\Gamma)=\CS^{m,k}(U;\Gamma)$.
\subsubsection{Symbols which are holomorphic in the parameter}
If $\Gamma=\R^n\times\Lambda$, where $\Lambda\subset\C$ is
a cone one may additionally require symbols to
be holomorphic in the $\Lambda$ variable. This aspect is
important if one deals with the resolvent of an elliptic differential
operator since the latter depends analytically on the resolvent parameter.
This class of symbols is not emphasized in this paper.

\subsubsection{Pseudodifferential operators with parameter}

Fix $a\in\sym^m(U;\R^n\times\Gamma)$ (respectively~$\in \CS^m(U;\R^n\times\Gamma)$).
For each fixed
$\mu_0\in\Gamma$ we have $a(\cdot, \cdot, \mu_0) \in \sym^m (U; \R^n)$
(respectively~$\in\CS^m(U;\R^n))$ and hence
we obtain a family of pseudodifferential operators
parametrized over $\Gamma$ by putting
\begin{equation}\label{eq:psido}
\begin{split}
 \big[ \Op( &a(\mu_0) ) \, u \big] \, (x):=  \big[ A(\mu_0) \, u \big] (x)\\
      &:= \int_{\R^n} \, e^{i \langle x,\xi \rangle} \, 
      a(x,\xi,\mu_0) \, \hat{u} (\xi ) \, \dbar \xi \\
      &= \int_{\R^n}\int_U \, e^{i \langle x-y,\xi \rangle} \, 
      a(x,\xi,\mu_0) \, u(y)  dy \dbar \xi .
\end{split}
\end{equation}

Note that the Schwartz kernel $K_{A(\mu_0)}$ of $A(\mu_0)=\Op(a(\mu_0))$
is given by
\begin{equation}\label{eq:Schwartz-kernel}
   K_{A(\mu_0)}(x,y,\mu_0)=\int_{\R^n}\, e^{i \scalar{x-y}{\xi} }\, 
      a(x,\xi,\mu_0) \, \dbar \xi .
\end{equation}
In general the integral is to be understood as an oscillatory integral,
for which we refer the reader to \cite{Shu:POST}, \cite{GriSjo:MAD}. 
The integral exists in the usual sense if $m+n<0$. 

The extension to manifolds and vector bundles is now straightforward.
Although historically it took quite a while until the theory
of singular integral operators had evolved into a theory
of pseudodifferential operators on vector bundles over smooth
manifolds (\textsc{Calder{\'o}n-Zygmund} \cite{CalZyg:SIO}, \textsc{Seeley}
\cite{See:SIC,See:IDO}, \textsc{Kohn-Nirenberg} \cite{KohNir:APD}).
For a smooth manifold $M$ and a vector bundle $E$ over $M$ we define the
space $\CL^m (M,E; \Gamma)$ of classical parameter dependent pseudodifferential
operators between sections of $E$ in the usual way by patching together local
data:

\begin{dfn}\label{d:pseudo-param}
Let $E$ be a complex vector bundle of finite fiber dimension $N$ over a smooth closed
manifold $M$ and let $\Gamma\subset\R^p$ be a cone. 
A {\em classical pseudo\-differential
operator of order $m$ with parameter} $\mu\in \Gamma$ is a family
of operators 
$B(\mu):\Gamma^\infty(M;E)\longrightarrow \Gamma^\infty(M;E),\, \mu\in\Gamma$, 
such that locally $B(\mu)$ is given by
\[
\big[B(\mu)\, u\big](x)=(2\pi)^{-n}\int_{\R^n}\int_Ue^{i\scalar{x-y}{\xi}} b(x,\xi,\mu)u(y)dyd\xi
\]
with $b$ an $N\times N$ matrix of functions belonging to
$\CS^m(U,\R^n\times \Gamma)$.

$\CL^{m,k}(M,E;\Gamma)$ is defined similarly, although we will discuss $\CL^{m,k}$
only in the non--parametric case. Of course, operators may act between
sections of different vector bundles $E,F$. In that case we write $\CL^{m,k}(M,E,F;\Gamma)$.
\end{dfn}

\begin{remark}\label{rem:20081120}

1. In case $\Gamma=\{0\}$ we obtain the usual (classical)
pseudodifferential operators of order $m$ on $U$. 
Here we write $\CL^m(M,E)$ instead of $\CL^m(M,E;\{0\})$
respectively $\CL^m(M,E,F)$ instead of $\CL^m(M,E,F;\{0\})$.

2. 
Parameter dependent pseudodifferential operators play a crucial
role, e.g., in the construction of the resolvent expansion
of an elliptic operator (\textsc{Gilkey} \cite{Gil:ITH}). 

A {\em pseudodifferential operator with parameter} is more than just a map from
$\Gamma$ to the space of pseudodifferential operators, cf. Corollary 
\ref{c:elliptic-regularity} and Remark
\ref{rem:elliptic-regularity}. 

To illustrate this let us consider a single elliptic operator
$A\in\CL^m(U)$. For simplicity let the symbol $a(x,\xi)$ of $A$
be positive definite. Then we can consider the ``parametric
symbol''
$b(x,\xi,\gl)=a(x,\xi)-\gl^m$ for $\gl\in \Lambda:=\C\setminus \R_+$.

However, in general $b$ lies in $\CS^m(U;\Lambda)$ only if $A$
is a differential operator. The reason is that $b$ will satisfy
the estimates \eqref{eq:3.1} only if $a(x,\xi)$
is polynomial in $\xi$, because then $\partial_\xi^\gb a(x,\xi)=0$
if $|\gb|>m$. If $a(x,\xi)$ is not polynomial in $\xi$, however,
\eqref{eq:3.1} will in general not hold if $\gb>m$.

This problem led \textsc{Grubb} and \textsc{Seeley} \cite{GruSee:WPP}
to invent their calculus of \emph{weakly parametric} pseudodifferential
operators. $b(x,\xi,\gl)=a(x,\xi)-\gl^m$ is weakly parametric
for any elliptic $A$ with positive definite leading symbol
(or more generally if $A$ satisfies Agmon's angle condition).
The class of weakly parametric operators is beyond the scope
of this survey, however.

3. The definition
of the parameter dependent calculus is not uniform in the literature.
It will be crucial in the sequel that differentiating by the parameter
reduces the order of the operator. This is the convention, e.g.
of \textsc{Gilkey} \cite{Gil:ITH} but differs from the one in 
\textsc{Shubin} \cite{Shu:POST}. 
In \textsc{Lesch--Pflaum} \cite[Sec.~3]{LesPfl:TAP} 
it is shown that parameter dependent pseudodifferential operators can
be viewed as translation invariant pseudodifferential
operators on $U\times \Gamma$ and therefore our convention
of the parameter dependent calculus contains \textsc{Melrose}'s
suspended algebra from \cite{Mel:EIF}. 
\end{remark}

\begin{prop}$\CL^{\bullet,\bullet}(M,E;\Gamma)$ is a bi--filtered algebra,
that is, 
\[A  B\in\CL^{m+m',k+k'}(M,E;\Gamma)\]
for $A\in\CL^{m,k}(M,E;\Gamma)$ and $B\in\CL^{m',k'}(M,E;\Gamma)$.
\end{prop}

The following result about the $L^2$--continuity of a
parameter dependent pseudodifferential operator is crucial.
We denote by $L^2_s(M,E)$ the Hilbert space of sections of $E$
of Sobolev class $s$.

\begin{theorem}\label{t:l2continuity}
Let $A\in\CL^m(M,E;\Gamma)$. Then for fixed $\mu\in\Gamma$ the
operator $A(\mu)$ extends by continuity to a bounded linear
operator $L^2_s(M,E)\longrightarrow L^2_{s-m}(M,E)$, $s\in\R$.

Furthermore, for $m\le 0$ one has the following uniform estimate
in $\mu$: for $0\le\vartheta\le 1, \mu_0\in\Gamma$, 
there is a constant $C(s,\vartheta)$ such that
\[
    \|A(\mu)\|_{s,s+\vartheta|m|}\le C(s,\vartheta,\mu_0) 
        (1+|\mu|)^{-(1-\vartheta)|m|},\quad |\mu|\ge |\mu_0|,\; \mu\in\Gamma.
\]
Here $\|A(\mu)\|_{s,s+\vartheta|m|}$ denotes the norm of the operator
$A(\mu)$ as a map from the Sobolev space $L^2_s(M,E)$ into
$L^2_{s+\vartheta |m|}(M,E)$.
\end{theorem}
If $\Gamma=\R^n$ then we can omit the $\mu_0$ in the formulation of the Theorem
(i.e. $\mu_0=0$). For a proof of Theorem \plref{t:l2continuity} see e.g.
\cite[Theorem 9.3]{Shu:POST}.

\subsubsection{The parametric leading symbol} 
The leading symbol of a classical pseudodifferential operator $A$ of order $m$
with parameter is now defined as follows: if $A$ has complete symbol $a(x,\xi,\mu)$
with expansion $a\sim\sum\limits_{j=0}^\infty a_{m-j}$ then
\begin{equation}\label{20081116-1}
  \begin{split}
     \sigma_A^m(x,\xi,\mu)&:=\lim_{r\to\infty} r^{-m}a(x,r\xi,r\mu)\\
                          &=  (|\xi|^2+|\mu|^2)^{m/2}
                             a_m(x,\frac{(\xi,\mu)}{\sqrt{|\xi|^2+|\mu|^2}}).
  \end{split}
\end{equation} 
$\sigma_A^m$ has an invariant meaning as a smooth function on
\[T^*M\times\gG\,\setminus\, \bigsetdef{(x,0,0)}{x\in M}\]
which is homogeneous in the following sense:
\[
\gs^m_A(x,r\xi,r\mu)=r^m\gs^m_A(x,\xi,\mu) \text{ for } (\xi,\mu)\ne (0,0),\, r>0.
\]

This symbol is determined by its restriction to the sphere in 
\[
S(T^*M\times \Gamma)=\bigsetdef{(\xi,\mu)\in T^*M\times \Gamma}{ |\xi|^2+|\mu|^2=1}
\]
and there is an exact sequence
\begin{equation}
0\longrightarrow \CL^{m-1}(M;\Gamma)\hookrightarrow \CL^m(M;\Gamma)\xrightarrow{\sigma}
C^\infty(S(T^*M\times\Gamma))\longrightarrow 0;
\end{equation}
the vector bundle $E$ being omitted from the notation just to save horizontal space.

\begin{example}\label{ex:20081120} Let us look at an example to illustrate the difference between
the parametric leading symbol and the leading symbol for a single pseudodifferential
operator. Let
\begin{equation}
    a(x,\xi)=\sum_{|\ga|\le m} a_\ga(x) \xi^\ga
\end{equation}
be the complete symbol of an elliptic \emph{differential} operator. Then
(cf. Remark \eqref{rem:20081120} 2.)
\begin{equation}
    b(x,\xi,\gl)= a(x,\xi) -\gl^m
\end{equation}
is a symbol of a parameter dependent (pseudo)differential operator $B(\gl)$
with parameter $\gl$ in a suitable cone $\Lambda\subset\C$. 
The parameter dependent leading symbol of $B$ is $\sigma_B^m(x,\xi,\gl)=a_m(x,\xi)-\gl^m$
while for fixed $\gl$ the leading symbol of the single operator $B(\gl)$ is
$\sigma_{B(\gl)}^m(x,\xi)=a_m(x,\xi)=\sigma_{B}^m(x,\xi,\gl=0)$.
\end{example}

In fact we have in general:

\begin{lemma} Let $A\in\CL^m(M,E;\Gamma)$ with parameter dependent leading symbol
$\sigma_A^m(x,\xi,\mu)$. For fixed $\mu_0\in\Gamma$ the operator $A(\mu_0)\in\CL^m(M,E)$
has leading symbol $\sigma_{A(\mu_0)}^m(x,\xi)=\sigma_A^m(x,\xi,0)$.
\end{lemma}
\begin{proof} It suffices to prove this locally in a chart $U$ for a scalar
  operator $A$. Since the leading
symbols are homogeneous it suffices to consider $\xi$ with $|\xi|=1$.

So suppose
that $A$ has complete symbol $a(x,\xi,\mu)$ in $U$. 
Write $a(x,\xi,\mu)=a_m(x,\xi,\mu)+\tilde a(x,\xi,\mu)$ 
with $\tilde a\in\CS^{m-1}(U;\R^n\times\Gamma)$ and $a_m(x,r\xi,r\mu)=r^m a_m(x,\xi,\mu)$
for $r\ge 1,|\xi|^2+|\mu|^2\ge 1$. 
Then for fixed $\mu_0\in\Gamma$ we have 
$\tilde a(\cdot,\cdot,\mu_0)\in\CS^{m-1}(U;\R^n)$ and hence
$\lim\limits_{r\to\infty} r^{-m}\tilde a(x,r\xi,\mu_0)=0$. 
Consequently
\begin{equation}\begin{split}
       \sigma_{A(\mu_0)}^m(x,\xi)&=\lim_{r\to\infty} r^{-m} a_m(x,r\xi,\mu_0)\\
              &=\lim_{r\to\infty} a_m(x,\xi,\mu_0/r)=a_m(x,\xi,0).\qedhere
  \end{split}
\end{equation}
\end{proof}

\subsubsection{Parameter dependent ellipticity} This is now defined as the invertibility 
of the parametric leading symbol.
The basic example of a pseudodifferential operator with parameter is the resolvent of an
elliptic differential operator (cf. Remark \ref{rem:20081120} and
Example \ref{ex:20081120}). The following two results can also be found
in \cite[Section II.9]{Shu:POST}.

\begin{theorem}\label{t:elliptic-regularity}
Let $M$ be a closed manifold and $E,F$ complex vector bundles over $M$.
Let $A\in\CL^m(M,E,F;\Gamma)$ be \emph{elliptic}. Then there exists a
$B\in\CL^{-m}(M,F,E;\Gamma)$ such that $AB-I\in\CL^{-\infty}(M,F;\Gamma)$, 
$BA-I\in \CL^{-\infty}(M,E;\Gamma)$. 
\end{theorem}

Note that in view of Theorem \plref{t:l2continuity}
this implies the estimates
\begin{equation}\label{eq:smoothing-estimate}
    \|B(\mu)A(\mu)-I\|_{s,t}+\|A(\mu)B(\mu)-I\|_{s,t}\le C(s,t,N) (1+|\mu|)^{-N}
\end{equation}
for all $s,t\in\R, N>0$. 
This result has an important implication:

\begin{cor}\label{c:elliptic-regularity}
Under the assumptions of Theorem \plref{t:elliptic-regularity}, for each
$s\in\R$ there is a $\mu_0\in\Gamma$ such that for $|\mu|\ge |\mu_0|$
the operator
\[
   A(\mu):L^2_s(M,E)\longrightarrow L^2_{s-m}(M,F)
\]
is invertible.
\end{cor}
\begin{proof}In view of \eqref{eq:smoothing-estimate} 
there is a $\mu_0=\mu_0(s)$ such that 
\[\|(BA-I)(\mu)\|_s<1 \text{ and } \|(AB-I)(\mu)\|_{s-m}<1,\]
for $|\mu|\ge |\mu_0|$ and hence $AB:L^2_s\longrightarrow L^2_s$ and 
$BA:L^2_{s-m}\longrightarrow L^2_{s-m}$
are invertible.
\end{proof}

\begin{remark}\label{rem:elliptic-regularity}
This result causes an interesting constraint on those pseudodifferential
operators which may appear as special values of an elliptic parametric family. Namely,
if $A\in\CL^m(M,E,F;\Gamma)$ is parametric elliptic then for each $\mu$ the operator
$A(\mu)\in\CL^m(M,E,F)$ is elliptic. Furthermore, by the previous Corollary and the
stability of the Fredholm index we have $\ind A(\mu)=0$ for all $\mu$.
\end{remark}

\section{Extending the Hilbert space trace to pseudodifferential
operators}
\label{s:EHS}

We pause the discussion of pseudodifferential operators and
look at the Hilbert space trace $\Tr$ on pseudodifferential operators.

\subsection{$\Tr$ on operators of order $<-\dim M$}\label{ss:coordinate-invariance}
Consider the local situation, i.e. a compactly
supported operator $A=\Op(a)\in\CL^{m,k}(U,E)$ in a local chart.

If $m<-\dim M$ then $A$ is trace class and the trace
is given by integrating the kernel of $A$ over the diagonal:
\begin{equation}\label{eq:20081121-1}
  \begin{split}
  \Tr(A)&=\int_U \tr_{E_x}\bigl(k_A(x,x)\bigr) dx\\
  &=\int_U \int_{\R^n} \tr_{E_x} \bigl(a(x,\xi)\bigr)\dbar\xi dx,
\end{split}
\end{equation}
where we have used \eqref{eq:Schwartz-kernel}.

The right hand side is indeed coordinate invariant.
To explain this consider
a coordinate transformation $\kappa:U\to V$. Denote
variables in $U$ by $x,y$ and variables in $V$ by $\tilde x,\tilde y$.
It is not so easy to write down the symbol of $\kappa_*A$.
However, an amplitude function (these
are ``symbols'' which depend on $x$ and $y$, otherwise the
basic formula \eqref{eq:psido} still holds)
for $\kappa_* A$ is given by
\begin{equation}\label{eq:symbol-coordinate-change}
     (\tilde x,\tilde y,\xi)\mapsto a(\kappa^{-1}\tilde x,
   \phi(\tilde x,\tilde y)^{-1}\xi) \frac{|\det D\kappa^{-1}(\tilde x,\tilde
   y)|}{|\det \phi(\tilde x,\tilde y)|},
\end{equation}
cf. \cite[Sec. 4.1, 4.2]{Shu:POST}, where $\phi(\tilde x,\tilde y)$ is
smooth with $\phi(\tilde x,\tilde x)=D\kappa^{-1}(\tilde x)^t$. 
Comparing the trace densities in the two coordinate systems requires
a \emph{linear} coordinate change in the $\xi$--variable.
Indeed,
\begin{equation}\label{eq:tr-coordinate-invariance}
\begin{split}
     \Tr(\kappa_*A)&=\int_V
               \int_{\R^n}\tr_{E_{\tilde x}}\bigl( a(\kappa^{-1}\tilde
	       x,\phi(\tilde x,\tilde x)^{-1}\xi)\bigr)\dbar\xi d\tilde x\\
   &=\int_V \int_{\R^n} \tr_{E_{\tilde x}}\bigl( a(\kappa^{-1}\tilde
   x,\xi)\bigr)\dbar\xi\,|\det D\kappa^{-1}(\tilde x)|d\tilde x,\\
   &=\int_U \int_{\R^n} \tr_{E_{x}}\bigl( a(x,x,\xi)\bigr)\dbar\xi\,dx=\Tr(A).
 \end{split}
\end{equation}

Therefore, the trace of a pseudodifferential operator $A\in\CL^{m,k}(M,E)$
of order $m<-\dim M=:-n$ on the closed manifold $M$ may be calculated from the complete symbol of
$A$ in coordinates as follows. Choose a finite open cover by coordinate
neighborhoods $U_j, j=1,\ldots, r,$ and a subordinated partition of unity
$\varphi_j, j=1,\ldots,r$. Furthermore, let $\psi_j\in\cinfz{U_j}$ with
$\psi_j\varphi_j=\varphi_j$. Denoting by $a_j(x,\xi)$ the complete symbol
in the coordinate system on $U_j$ we obtain
\begin{equation}\label{eq:20090514}
\Tr(A)=\sum_{j=1}^r \Tr(\varphi_j A\psi_j)=\sum_{j=1}^r
\int_{U_j}\int_{\R^n} \varphi_j(x) \tr_{E_x}\bigl(a_j(x,\xi)\bigr)\dbar\xi\,dx.
\end{equation}

\begin{sloppy}
A priori the previous argument is valid only for operators of order $m<-n$. 
However, the symbol function $a_j(x,\xi)$ is rather well--behaved in
$\xi$. If for a class of pseudodifferential operators
we can regularize $\int_{\R^n} a_j(x,\xi) \dbar\xi $ in such a way 
that the change of variables
\eqref{eq:tr-coordinate-invariance} works then indeed \eqref{eq:20090514}
extends the trace to this class of operators. Such a regularization is provided by:
\end{sloppy}

\subsection{The Hadamard partie finie regularized integral}\label{ss:partie-finie}


The problem of regularizing divergent integrals is in fact quite old.
The method we are going to present here goes back to \textsc{Hadamard} who used
his method to regularize integrals which arose when solving the wave equation
\cite{Had:PCE}.

Given a function $f\in \CS^{m,k}(\R^p)$, e.g. $a(x,\cdot)$ above for fixed
$x$. Then $f$ has an asymptotic expansion
\begin{equation}\label{eq:20081120-6}
   f(x)\sim_{|x|\to\infty}
         \sum_{j=0}^\infty \sum_{l=0}^{k} f_{jl}(x/|x|)|x|^{m-j}\log^l|x|.
       \end{equation}
Integrating over balls of radius $R$ gives the asymptotic expansion
\begin{equation}\label{eq:20081120-7}
   \int_{|x|\le R} f(x) dx \sim_{R\to\infty}
   \sum_{j=0}^\infty \sum_{l=0}^{k+1} \tilde f_{jl} R^{m+n-j} \log^l R.
 \end{equation}
The \emph{regularized integral}
$\displaystyle \regint_{\R^p} f(x) dx$ is, by definition, the 
constant term in this asymptotic expansion. Some authors call
the regularized integral \emph{partie finie integral} or 
\emph{cut--off integral}.

It has a couple of peculiar
properties, cf.~\cite{Mel:EIF}, which were further investigated in 
\cite[Sec.~5]{Les:NRP} and \cite{LesPfl:TAP}. 
The most notable features are a modified change of
variables rule for linear coordinate changes
and, as a consequence, the fact that Stokes' theorem does
not hold in general:

\begin{prop}\textup{\cite[Prop.~5.2]{Les:NRP}}\label{p:change-variables}
Let $A\in {\rm GL}(p,\R)$ be a regular matrix. Further\-more, let
$f \in\CS^{m,k}(\R^p)$ with expansion \eqref{eq:20081120-6}.
Then we have the change of variables formula
\begin{multline}
\regint_{\R^p} f(A\xi) d\xi\\=
   |\det A|^{-1}\left( \regint_{\R^p} f(\xi)d\xi+
      \sum_{l=0}^{k}\frac{(-1)^{l+1}}{l+1}\int_{S^{p-1}} f_{-p,l}(\xi)
          \log^{l+1} |A^{-1}\xi| d\xi\right).
	\end{multline}
\end{prop}

\commentary{, or
in other words $\reginttext$
is not a closed functional on $\gO^*({\rm PS}(\R^p))$. 
More precisely, we extend $\reginttext$ to $\gO^* ({\rm PS}^*(\R^p))$
by putting 
\begin{equation}
  \regint:\go\mapsto\casetwo{0}{\go\in \gO^k,k<p,}{
    \regint_{\R^p}f(\xi)d\xi}{\go=f(\xi) d\xi_1\wedge\ldots\wedge d\xi_p.}
\end{equation}
In this way we obtain a graded trace on the complex 
$(\gO^* ({\rm PS}^*(\R^p)),d)$. This would be a cycle in the sense of
{\rm Connes}  \cite[Sec. III.1.$\alpha$]{Con:NG} if $\reginttext$
were closed. 

The next lemma shows that $d\reginttext$, which
is defined by $(d\reginttext)\go:=\reginttext d\go$, 
is nontrivial. However, it is local in the sense that it depends
only on the $\log$--polyhomogeneous expansion of $\go$.}

The following proposition, which substantiates the mentioned fact that
Stokes' Theorem does not hold for $\reginttext$, was stated as a Lemma in \cite{LesPfl:TAP}.
A couple of years later 
it was rediscovered by \textsc{Manchon}, \textsc{Maeda}, and
\textsc{Paycha} \cite{Manetal:SFC}, \cite{Pay:NRC}.

\begin{prop}\textup{\cite[Lemma 5.5]{LesPfl:TAP}} \label{S2-4.4} 
  Let $f\in \CS^{m,k}(\R^p)$
  with asymptotic expansion \eqref{eq:20081120-6}.
Then
\[\regint_{\R^p} \frac{\pl f}{\pl \xi_j} d\xi=
   \int_{S^{p-1}} f_{1-p,k}(\xi) \xi_j d{\rm vol}_S(\xi).\]
\end{prop}

We will come back to this below when we discuss the residue trace.

\subsection{The Kontsevich--Vishik canonical trace}

Using the Hadamard partie finie integral we can now follow the
scheme outlined in Subsection \ref{ss:coordinate-invariance}. 
Let $A\in\CL^{a,k}(M,E)$ be a $\log$--polyhomogeneous pseudodifferential
operator on a closed manifold $M$. If $a\not\in\Z$ we put, using
the notation of \eqref{eq:20090514} and
\eqref{eq:tr-coordinate-invariance},
\begin{equation}
  \TR(A):=\sum_{j=1} \int_{U_j}\regint_{\R^n}\varphi_j(x) \tr_{E_x}\bigl(
  a_j(x,\xi)\bigr)\dbar\xi\,dx.
\end{equation}
By Proposition \plref{p:change-variables} one shows exactly as in
\eqref{eq:tr-coordinate-invariance} that $\TR(A)$ is well--defined.

In fact we have (essentially) proved the following:

\begin{theorem}[\textsc{Kontsevich--Vishik}  {\cite{KonVis:GDE},
\cite{KonVis:DEP}},\newline \textsc{Lesch} {\cite[Sec. 5]{Les:NRP}}]\label{t:kont-vish}
There is a linear functional $\TR$ on 
$$\bigcup_{a\in \C\setminus\{-n,-n+1,-n+2,\ldots\},k\ge 0} \CL^{a,k}(M,E)$$
such that
\begin{enumerate}\renewcommand{\labelenumi}{{\rm (\roman{enumi})}}
  \item In a local chart $\TR$ is given by \eqref{eq:20081121-1}, with
  $\int_{\R^n}$ to be replaced by the cut--off integral $\regint_{\R^n}$.
\item $\TR\restriction\CL^{a,k}(M,E)=\Tr\restriction\CL^{a,k}(M,E)$ if $a<-\dim M$.
\item $\TR([A,B])=0$ if $A\in\CL^{a,k}(M,E), B\in CL^{b,l}(M,E)$, $a+b\not\in\Z$.
\end{enumerate}
\end{theorem}

We mention a stunning application of this result
\cite[Cor. 4.1]{KonVis:GDE}. Let $G$ be a domain in the complex
plane and let $A(z),B(z)$ be holomorphic families of operators
in $\CL^{\bullet,k}(M,E)$ with $\ord A(z)=\ord B(z)=z$. We do not
formalize the notion of a holomorphic family here. What we have
in mind are e.g. families of complex powers $A(z)=A^z$. Assume
that $G$ contains points $z$ with $\Re z<-\dim M$.
Then $\TR(A(z))$ is the analytic continuation of 
$\Tr(A(\cdot))\restriction G\cap \bigsetdef{z\in\C}{\Re z<-\dim M}$;
a similar statement
holds for $B(z)$.

If for a point $z_0\in G\setminus \{-n,-n+1,\dots\}$ 
we have $A(z_0)=B(z_0)$ we can conclude
that the value of the analytic continuation of
$\Tr(A(\cdot))\restriction G\cap \bigsetdef{z\in\C}{\Re z<-\dim M}$
to $z_0$ coincides with the value of the corresponding
analytic continuation of $\Tr(B(\cdot))\restriction G\cap
\bigsetdef{z\in\C}{\Re z<-\dim M}$.
Namely, we obviously have $\TR(A(z_0))=\TR(B(z_0))$.
The author does not know of a direct proof of this fact.

Proposition \plref{p:change-variables} shows that if $A$ is of integral
order additional terms show up when making the linear change of coordinates
\eqref{eq:tr-coordinate-invariance}, indicating that $\TR$ cannot
be extended to a trace on the algebra of pseudodifferential operators.
The following no go result shows that the order constraints in Theorem
\plref{t:kont-vish} are indeed sharp:



\begin{prop}\label{p:no-trace} There is no trace $\tau$ on the algebra
$\CL^0(M)$ of classical pseudodifferential operators of order $0$
such that $\tau(A)=\Tr(A)$ if $A\in\CL^{-\infty}(M)$.
\end{prop}
\begin{proof} We reproduce here the very easy proof: from Index Theory
  we use the fact that on $M$ there exists an elliptic system
  $T\in \CL^0(M,\C^r)$ of non--vanishing Fredholm index; in general
  we cannot find a scalar elliptic operator with non--trivial
  index. Let $S\in\CL^0(M,\C^r)$ be a pseudodifferential parametrix (cf.
  Theorem \ref{t:elliptic-regularity}) such
  that $I-ST, I-TS\in \CL^{-\infty}(M,\C^r)$. $\tau$ and
  $\Tr$ extend to traces on 
  $\CL^0(M,\C^r)=\CL^0(M)\otimes \operatorname{M}(r,\C)$ 
  via $\tau(A\otimes X)=\tau(A)\Tr(X)$, $A\in\CL^a(M),X\in\operatorname{M}(r,\C)$
  and $\Tr(X)$ is the usual trace on matrices.
  Since smoothing operators are of trace class one has
  \begin{equation}\label{eq:index-trace-formula}
    \ind T =\Tr(I-ST)-\Tr(I-TS)
  \end{equation}
  and we arrive at the contradiction
\begin{equation}
  \begin{split}
   0&\not=\ind T=\Tr(I-ST)-\Tr(I-TS)\\&=\tau(I-ST)-\tau(I-TS)=\tau([T,S])=0.\qedhere
   \label{eq:trace-contradiction}
 \end{split}
\end{equation}
\end{proof}

\section{Pseudodifferential operators with parameter: Asymptotic expansions}
\label{s:POPAE}

We take up Section \plref{s:POP} and continue the discussion of
pseudodifferential operators with parameter.

\subsection{The Resolvent Expansion} 

The following result is the main technical result needed for the residue
trace. It goes back to \textsc{Minakshisundaram} and \textsc{Pleijel}
\cite{MinPle:SPE} who follow carefully \textsc{Hadamard}'s method
of the construction of a fundamental solution for the wave equation
\cite{Had:PCE}. It is at the heart of the Local Index Theorem and
therefore has received much attention.
In the form stated below it is essentially due to \textsc{Seeley} \cite{See:CPE},
see also \cite{GruSee:WPP}. 
The (straightforward) generalization to $\log$--polyhomogeneous
symbols was done by the author \cite{Les:NRP}. Of the latter the published
version contains annoying typos, the arxiv version is correct.

\begin{theorem}\label{t:parametric-expansion}

\textup{1.} Let $U\subset\R^n$ open, $\Gamma\subset\R^p$ a cone, and 
$a\in\CS^{m,k}(U;\Gamma),$ $m+n<0$,
$A=\Op(a)$. Let $k_A(x;\mu):=\int_{\R^n}a(x,\xi,\mu)\dbar\xi$ be
the Schwartz kernel (cf. Eq. \eqref{eq:Schwartz-kernel}) 
of $A$ on the diagonal. Then 
$k_A\in\CS^{m+n,k}(U;\Gamma)$. In particular there is an asymptotic expansion
\begin{equation}\label{eq20081120-1}
  k_A(x,x;\mu)\sim_{|\mu|\to\infty}\sum_{j=0}^\infty\sum_{l=0}^k e_{m-j,l}(x,\mu/|\mu|) |\mu|^{m+n-j}\log^k|\mu|.
\end{equation}

\textup{2.} Let $M$ be a compact manifold,
$\dim M=:n$, and $A \in \CL^{m,k} (M,E;\Gamma)$. 
If $m + n< 0$ then $A(\mu)$ is trace class for all $\mu \in \Gamma$ and
$ \Tr \, A (\cdot) \in \CS^{m + n,k} (\Gamma)$.
In particular, 
\[\Tr \, A(\mu)\sim_{|\mu|\to\infty}
\sum\limits_{j=0}^\infty\sum\limits_{l=0}^k e_{m-j,l}(\mu/|\mu|) |\mu|^{m+n-j}\log^k|\mu|.
\]

\textup{3.} Let $P\in\CL^m(M,E)$ be an elliptic classical pseudodifferential operator 
and assume for simplicity that with respect to some Riemannian structure on $M$ and 
some Hermitian structure on $E$
the operator $P$ is self--adjoint and non--negative. Furthermore, let $B\in\CL^{b,k}(M,E)$
be a pseudodifferential operator. Let $\Lambda=\bigsetdef{\gl\in\C}{|\arg\gl|\ge \eps}$
be a sector in $\C\setminus\R_+$. Then for $N>(b+n)/m, n:=\dim M,$ the
operator $B(P-\gl)^ {-N}$ is of trace class and there is an asymptotic expansion
\begin{equation}\label{eq20081120-3}\begin{split}
\Tr (B(P-\lambda)^{-N})\sim_{\lambda\to \infty}&\sum_{j=0}^\infty\sum_{l=0}^{k+1} c_{jl} 
\lambda^{\frac{n+b-j}{m}-N}\log^l \gl+\\
    &+ \sum_{j=0}^\infty d_j\; \lambda^{-j-N}
  \end{split},\quad \gl\in\Lambda.
\end{equation}
Furthermore, $c_{j,k+1}=0$ if $(j-b-n)/m\not\in\Z_+$.
\end{theorem}
\begin{proof} 
We present a proof of 1. and 2. and sketch the proof of 3. in a special case.

Since $a \in \CS^{m,k} (U;\Gamma)$ we have Eq. \eqref{ML-G2.2}.
Thus we write
  \begin{equation}
     a = \sum_{j=0}^{N} \, a_{m-j} \, + R_N,
  \end{equation}
  with $R_N \in \sym^{m-N} (U;\Gamma)$.
  In fact, $R_N\in \sym^{m-N-1+\eps}(U;\Gamma)$ for every $\eps>0$,
  but we don't need this below.
  Now pick $L \subset \Gamma, K\subset U,$ compact and a multi--index $\alpha$. Then
  for $x\in K$ the kernel $k_{A,N}$ of $R_N$ satisfies
  \begin{equation}\label{eq20081120-2}
  \begin{split}
    \Bigl| \partial^{\alpha}_{\mu} &k_{A,N}(x,x;\mu) \Bigr|\\
    & = \Bigl| \int_{\R^n}  \partial_{\mu}^{\alpha} R_N(x , \xi , \mu ) 
                 \dbar \xi  \Bigr| \\
    & \leq C_{\alpha, K , L}  \int_{\R^n} ( 1 + ( |\xi|^2
      +|\mu|^2)^{1/2} )^{m -|\alpha| - N} \, \dbar\xi \\
    & \leq C_{\alpha,K,L} (1 + |\mu|)^{m+n - |\alpha| - N }.
  \end{split}
  \end{equation}
  Now consider one of the summands of \eqref{ML-G2.2}. We write
  it in the form 
  \begin{equation}
  b_{m-j,l}(x,\xi,\mu)=\tilde b_{m-j,l}(x,\xi,\mu) \log^l(|\xi|^2+|\mu|^2),
  \end{equation}
  with
  \begin{equation}
  \tilde b_{m-j,l}(x,r\xi,r\mu)=r^{m-j}\tilde b_{m-j,l}(x,\xi,\mu),
  \quad \text{ for }
  r\ge 1, |\xi|^2+|\mu|^2\ge 1.
  \end{equation}
  Then the contribution $k_{m-j,l}$
  of $b_{m-j,l}$ to the kernel of $A$ satisfies
  \begin{equation}
  \begin{split}
    k_{m-j,l}&(x,x;r \mu)\\
      & = \int_{\R^n}
    \tilde b_{m-j,l}(x, \xi, r\mu ) \,\log^l(|\xi|^2+r^2|\mu|^2)\, \dbar \xi\\
    & = r^{m-j} \, \int_{\R^n} \tilde b_{m-j,l}(x , r^{-1} \xi ,\mu ) \bigl(\log r^2+
         \log(|r^{-1}\xi|^2+|\mu|^2)\bigr)^l\,\dbar \xi \\
	 & = r^{m+n-j} \int_{\R^n} \tilde b_{m-j,l}(x,\xi,\mu) \bigl(\log r^2+
	 \log(|\xi|^2+|\mu|^2)\bigr)^l\,\dbar \xi,
  \end{split}
  \end{equation}
  proving the expansion \eqref{eq20081120-1}.

  2. follows simply by integrating \eqref{eq20081120-1}. In view of   
  \eqref{eq20081120-2} the expansion \eqref{eq20081120-1} is uniform
  on compact subsets of $U$ and hence may be integrated over compact
  subsets. Covering the compact manifold $M$ by finitely many charts
  then gives the claim.

3. We cannot give a full proof of 3. here; but we at least want to explain
where the additional $\log$ terms in \eqref{eq20081120-3} come from.
Note that even if $B\in\CL^b(M,E)$ is classical there are $\log$ terms
in \eqref{eq20081120-3}. In general the highest $\log$ power occurring
on the rhs of \eqref{eq20081120-3} is one higher than the $\log$ degree
of $B$.


For simplicity let us assume that $P$ is a differential operator. This
ensures that $(P-\gl^m)^{-N}$ (note the $\gl^m$ instead of $\gl$) is in the parametric calculus
(cf. Remarks \plref{rem:20081120} 2., \plref{ex:20081120}). 
We first describe the local expansion of the symbol of $B(P-\gl^m)^{-N}$.
To obtain the claim as stated one then has to replace $\gl^m$ by $\gl$
and integrate over $M$:
choose a chart and denote the complete symbol of $B$ by $b(x,\xi)$
and the complete parametric symbol of $(P-\gl^m)^{-N}$ by
$q(x,\xi,\gl)$. Then the symbol of the product is given by
\begin{equation}\label{eq:product-symbol}
  (b*q)(x,\xi,\gl)\sim\sum_{\ga\in\Z_+^n} \frac{i^{-\ga}}{\ga!}
  \bigl(\partial_\xi^\ga b(x,\xi)\bigr)\bigl(\partial_x^\ga
  q(x,\xi,\gl)\bigr).
\end{equation}
Expanding the rhs into its homogeneous components gives
\begin{equation}\label{eq:product-symbol-a}
  \begin{split}
  (b&*q)(x,\xi,\gl)\\
  &\sim\sum_{j=0}^\infty\sum_{|\ga|+l+l'=j} \frac{i^{-\ga}}{\ga!}
  \underbrace{\underbrace{\bigl(\partial_\xi^\ga b_{b-l}(x,\xi)\bigr)}_{(b-l-|\ga|)-\text{(log)homogeneous}}
  \underbrace{\bigl(\partial_x^\ga
  q_{-mN-l'}(x,\xi,\gl)\bigr)}_{(-mN-l')-\text{homogeneous}}
  }_{(b-mN-j)-\text{(log)homogeneous}}.
\end{split}
\end{equation}
The contribution to the Schwartz kernel of $B(P-\gl^m)^{-N}$ of a summand
is given by
\begin{equation}\label{eq20081120-4}
    \frac{i^{-\ga}}{\ga!} \int_{\R^n}
  \bigl(\partial_\xi^\ga b_{b-l}(x,\xi)\bigr) \bigl(\partial_x^\ga
  q_{-mN-l'}(x,\xi,\gl)\bigr)\, \dbar\xi.  
\end{equation}
We will see that the asymptotic expansion of each of these integrals
a priori contributes to the term $\gl^{-N}$ in the expansion \eqref{eq20081120-3}. So
additional considerations, which we will not present here, are
necessary to show that by expanding the individual integrals
\eqref{eq20081120-4} one indeed obtains the asymptotic expansion
\eqref{eq20081120-3}.

The asymptotic expansion of \eqref{eq20081120-4} will be singled out
as Lemma \plref{l:expansion-lemma} below. The proof of it will in
particular explain why the highest possible $\log$-power in
\eqref{eq20081120-3} is one higher than the $\log$-degree of $B$
\end{proof}

The following expansion Lemma is maybe of interest in its own right. 
Its proof will explain the occurrence of higher $\log$ powers in the resolvent respectively
heat expansions. The homogeneous version of the Lemma can again be found
in \cite{GruSee:WPP}. We generalize it here slightly to the
$\log$--polyhomogeneous setting (cf. \cite{Les:NRP}).

\begin{lemma}\label{l:expansion-lemma}
Let $B\in\cinf{\R^n}, Q\in\cinf{\R^n\times [1,\infty)}$
and assume that $B,Q$ have the following properties
\begin{equation}
 \begin{split}
    B(\xi)&= \tilde B(\xi/|\xi|) |\xi|^b \log^k|\xi|,\quad  |\xi|\ge 1,\\
    Q(r\xi,r\gl)&=r^q Q(\xi,\gl),\quad r\ge 1, \gl\ge 1, \\
    |Q(\xi,1)| &\le C (|\xi|+1)^{-q},
\end{split}
\end{equation}
where $b,q\in\R$ and $b+q+n<0$. 
Then the following asymptotic expansion holds:
\begin{equation}\label{eq:expansion-lemma}
  \begin{split}
      F(\gl)&= \int_{\R^n} B(\xi)Q(\xi,\gl) d\xi\\
   &\sim_{\gl\to\infty} \sum_{j=0}^{k+1} c_j \gl^{q+b+n}\log^j\gl +\sum_{j=0}^\infty d_j \gl^{q-j}.
  \end{split}
\end{equation}
$c_{k+1}=0$ if $b$ is not an integer $\le -n$.

The coefficients $c_j,d_j$ will be explained in the proof.
\end{lemma}
\begin{proof}
The integral on the lhs of \eqref{eq:expansion-lemma} exists since
$b+q+n<0$.

We split the domain of integration into the three regions:\\
$1\le \gl\le |\xi|, |\xi|\le 1,$ and $1\le |\xi|\le \gl$.

\paragraph*{$ 1\le\gl\le |\xi|$:} Here we are in the domain of homogeneity
and a change of variables yields
\begin{equation}
\begin{split}
             &\int_{\gl\le |\xi|} B(\xi)Q(\xi,\gl) d\xi\\
       &= \gl^{q}  \int_{\gl\le|\xi|} 
               \tilde B(\xi/|\xi|) |\xi|^b \bigl(\log^k|\xi|\bigr) Q(\xi/\gl,1) d\xi\\
        &=\gl^{q+b+n} \int_{1\le |\xi|} 
       \tilde B(\xi/|\xi|) |\xi|^b \bigl(\log\gl+\log |\xi|\bigr)^kQ(\xi,1) d\xi,\\
    &=\sum_{j=0}^k \ga_j \gl^{q+b+n}\log^j\gl,
\end{split}
\end{equation}
giving a contribution to the coefficient $c_j$ for $0\le j\le k$.
\paragraph*{$ |\xi|\le 1$:} For the remaining two cases we employ
the Taylor expansion of the smooth function $\eta\mapsto Q(\eta,1)$
about $\eta=0$:
\begin{equation}\label{eq:Taylor-q}
   Q(\eta,1)= \sum_{j=0}^N Q_j(\eta) +R_N(\eta),
\end{equation}
where $Q_j(\eta)\in\C[\eta_1,,\ldots,\eta_n]$ are homogeneous polynomials of degree
$j$ and $R_N$ is a smooth function satisfying $R_N(\eta)=O(|\eta|^{N+1}),\;\eta\to 0$.
Respectively, for $\xi\in\R^n, \gl\ge 1$,
\begin{equation}\label{eq:Taylor-ql}
   Q(\xi,\gl)=Q(\xi/\gl,1)\;\gl^q =\sum_{j=0}^N Q_j(\xi)\;\gl^{q-j} +
R_N(\xi/\gl)\;\gl^q.
\end{equation}

Plugging \eqref{eq:Taylor-ql} into the integral for $|\xi|\le 1$ we find
\begin{equation}
   \begin{split}
      \int_{|\xi|\le 1} &B(\xi) Q(\xi,\gl)d\xi=\\
     &=\sum_{j=0}^N \int_{|\xi|\le 1} B(\xi)Q_j(\xi) d\xi\; \gl^{q-j}  +O(\gl^{q-N-1}),
\quad \gl\to\infty,
   \end{split}
\end{equation}
giving a contribution to the coefficient $d_j$.

\paragraph*{$1\le |\xi|\le \gl$:} We again use the Taylor expansion
\eqref{eq:Taylor-ql} with $N$ large enough such that $b+N+1>-n$
to ensure  $\int_{|\xi|\le 1} |\xi|^b \log^j|\xi| \;|R_N(\xi)|d\xi<\infty$ for all $j$.
Let $B^h(\xi):=\tilde B(\xi/|\xi|) |\xi|^b\log^k|\xi|$ be the homogeneous extension 
of $B(\xi)$ to all $\xi\not=0$.
Then
\begin{equation}
    \int_{|\xi|\le 1} \bigl(|B(\xi)|+|B^h(\xi)|\bigr) \gl^q |R_N(\xi/\gl)|d\xi
       =O(\gl^{q-N-1}),\quad \gl\to\infty,
\end{equation}
and thus
\begin{equation}
   \begin{split}
&\int_{1\le |\xi|\le \gl} B(\xi) \gl^q R_N(\xi/\gl) d\xi\\
     &= \int_{0\le |\xi|\le \gl} B^h(\xi) \gl^q R_N(\xi/\gl)d\xi+O(\gl^{q-N-1})\\
     &= \int_{|\xi|\le 1} \tilde B(\xi/|\xi|) |\xi|^b \bigl(\log\gl+\log|\xi|\bigr)^k R_N(\xi) d\xi\; \gl^{q+b+n}+\\
    &\quad +O(\gl^{q-N-1}),
\quad \gl\to\infty.
   \end{split}
\end{equation}
So the contribution of the ``remainder'' $R_N$ to the expansion is not
small, rather it contributes to the coefficient $c_j$ of the
$\gl^{q+b+n}\log^j\gl$ term for $0\le j\le k$. Note that so far we have not obtained
any contribution to the coefficient $c_{k+1}$. 

Such a contribution will show up only now when
we finally deal with the summands in the Taylor expansion. 
Using polar coordinates we find
\begin{equation}\label{eq:expansion-lemma-proof}
   \begin{split}
      &\int_{1\le |\xi|\le \gl} B(\xi) Q_j(\xi)d\xi\; \gl^{q-j}\\
        &= \gl^{q-j}\int_1^\gl \int_{S^{n-1}} \tilde B(\go) r^b \bigl(\log^k r\bigr)
               Q_j(r\go) r^{n-1}d\vol_{S^{n-1}}(\go) dr \\
       &= C_j \gl^{q-j} \int_1^\gl r^{b+n-1+j}\log^k r dr\\
       &=C_j \gl^{q-j} \begin{cases} 
       \sum\limits_{\sigma=0}^k \ga'_\sigma\gl^{b+n+j} \log^\sigma\gl+\beta_j ,& b+n+j\not=0,\\[1em]
       \frac{1}{k+1}\log^{k+1}\gl, &     b+n+j=0.
                       \end{cases}
   \end{split}
\end{equation}
As a side remark note the explicit formula
\begin{multline}\label{eq:log-int-explicit}
   \int_1^\gl r^\ga \log^k r dr\\
= \begin{cases}
\sum\limits_{j=0}^k \frac{(-1)^j k!}{(k-j)!(\ga+1)^{j+1}} \gl^{\ga+1}\log^{k-j}\gl+\frac{(-1)^{k+1}k!}{(\ga+1)^{k+1}},& \ga\not=-1,\\
\frac{1}{k+1}\log^{k+1}\gl,                 & \ga=-1.
\end{cases}
\end{multline}
The constant term in \eqref{eq:log-int-explicit} respectively $\beta_j$ on the rhs of
\eqref{eq:expansion-lemma-proof} was omitted in \cite[Eq. 3.16]{Les:NRP}.
Fortunately the error was inconsequential for the formulation of the expansion result
because $\beta_j$ is just another contribution to the coefficient $d_j$.
\end{proof}

\subsection{Resolvent expansion vs. heat expansion}
\label{ss:resolvent-expansion}

\setlength{\unitlength}{1.0cm}
\begin{figure}
\begin{picture}(5.0,5.0)
\put(2.5,0){\vector(0,1){5.0}}   
\put(0,2.5){\vector(1,0){5.0}}
\put(2.5,2.5){\linethickness{1mm}\line(1,0){2.5}}

\put(2.5,3.5){\vector(1,1){1.5}} 
\put(2.5,1.5){\line(1,-1){1.5}}  
\qbezier(2.5,1.5) (1.5,2.5) (2.5,3.5) 
\end{picture}
\caption{
\label{figureone} Contour of integration for calculating $Be^{-tP}$ from 
the resolvent.}
\end{figure}
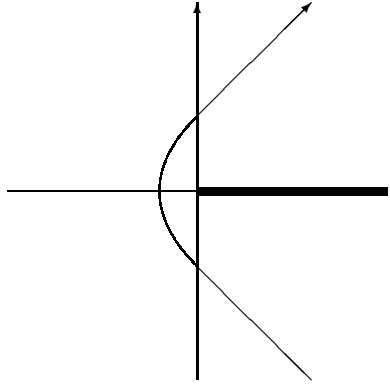

From the resolvent expansion one can easily derive the heat expansion
and the meromorphic continuation of the $\zeta$--function. In fact
under a mild additional assumption the resolvent expansion can be
derived from the heat expansion of the meromorphic continuation of 
the $\zeta$--function (cf. e.g. \textsc{Lesch} \cite[Theorem 5.1.4 and 5.1.5]{Les:OFT},
\textsc{Br\"uning--Lesch} \cite[Lemma 2.1 and 2.2]{BruLes:EIC}).

Let $B,P$ be as above. Next let $\gamma$ be a contour in the complex plane
as sketched in Figure \ref{figureone}. Then $B e^{-tP}$ has the following contour integral
representation:

\begin{equation}\label{eq:heat-contour-rep}
\begin{split}
    B e^{-tP}&= \frac{-1}{2\pi i} \int_\gamma e^{-t\gl} B(P-\gl)\ii d\gl\\
        &= -(-t)^{-N+1} \frac{(N-1)!}{2\pi i}\int_\gamma e^{-t\gl} B(P-\gl)^{-N}d\gl.
\end{split}
\end{equation}

Taking the trace on both sides and plugging in the asymptotic expansion
of $\Tr(B(P-\gl)^{-N})$ one easily finds
\begin{equation}\label{eq:log-heat-expansion}
  \Tr(B e^{-tP})\sim_{t\to 0+}\sum_{j=0}^\infty\sum_{l=0}^{k+1} a_{jl}(B,P) t^{\frac{j-b-n}{m}}\log^l t
     +\sum_{j= 0}^\infty  \tilde d_j(B,P)\; t^j.
\end{equation}
$a_{j,k+1}=0$ if $(j-b-n)/m\not\in\Z_+$.

\subsection{Heat expansion vs. $\zeta$--function}
\label{ss:heat-zeta}

Finally we briefly explain how the meromorphic continuation of the
$\zeta$--function can be obtained from the heat expansion. As before
let $B\in\CL^{b,k}(M,E)$ and let $P\in\CL^m(M,E)$ be an elliptic
operator which is self--adjoint with respect to some
Riemannian structure on $M$ and some Hermitian structure on $E$.
Furthermore, assume that $P\ge 0$ is non--negative.
Let $\Pi_{\ker P}$ be the orthogonal projection onto $\ker P$
and put for $\Re s>0$
\begin{equation}\label{eq:20090511-4}
P^{-s}:= \bigl(I-\Pi_{\ker P}\bigr)\bigl(P+\Pi_{\ker P}\bigr)^{-s}.
\end{equation}
I.e. $P^{-s}\restriction \ker P=0$ and for $\xi\in\im P$
we let $P^{-s}\xi$ be the unique $\eta\in \ker P^\perp$
with $P^s\eta=\xi$.
The $\zeta$--function of $(B,P)$ is defined
(up to a $\Gamma$--factor) as the \emph{Mellin transform} of
the heat trace $\Tr(B (I-\Pi_{\ker P})e^{-tP})$:
\begin{equation}
  \begin{split}
       \zeta(B,P;s)&=\Tr\bigl(B P^{-s}\bigr)\\
                   &= \frac{1}{\Gamma(s)}\int_0^\infty t^{s-1} \Tr\bigl(B
(I-\Pi_{\ker P})e^{-tP}\bigr) dt,\quad \Re s\gg 0.
  \end{split}
  \label{eq:zeta-function}
\end{equation}
$\Tr\bigl(B(I-\Pi_{\ker P}) e^{-tP}\bigr)$ decays exponentially as
$t\to\infty$.
The meromorphic continuation is thus obtained by plugging the short time
asymptotic expansion \eqref{eq:log-heat-expansion} into the
rhs of \eqref{eq:zeta-function} (cf. e.g. \cite[Sec. II.1]{Les:OFT}):
\begin{equation}
  \begin{split}
  \Gamma(s)\zeta(B,P;s)&=\int_0^1 t^{s-1} \Tr(B e^{-tP}) dt \\
   &\qquad-\frac 1s\Tr\bigl(B\Pi_{\ker P}\bigr)+\text{ Entire function}(s),\\
                       &\sim\sum_{j=0}^\infty \sum_{j=0}^{k+1} 
		           \frac{a_{jl}'(B,P)}{(s-\frac{n+b-j}{m})^{j+1}}+\sum_{j=0}^\infty
\frac{\tilde d_j'(B,P)}{s+j},
  \end{split}	
  \label{eq:zeta-pole-structure}
\end{equation}
where the formal sum on the right is meant to display the principal
parts of the Laurent series at the poles of $\Gamma(s)\zeta(B,P;s)$.

The $\Gamma$--function has simple poles in $\Z_-=\{0,-1,-2,\dots\}$, hence
the $\tilde d_j'$ do not contribute to the poles of $\zeta(B,P;s)$. The $a_{jl}'$ depend linearly on the $a_{jl}$ and consequently
$a_{j,k+1}'=0$ if $(n+b-j)/m$ is \emph{not} a pole of the
$\Gamma$--function. Let us summarize:

\begin{theorem}\label{t:zeta-meromorphic} Let $M$ be a compact closed
  manifold of dimension $n$. 
  Let  $B\in\CL^{b,k}(M,E)$ and let $P\in\CL^m(M,E)$ be an elliptic
operator which is self--adjoint with respect to some
Riemannian structure on $M$ and some Hermitian structure on $E$.
Then the $\zeta$--function $\zeta(B,P;s)$ is meromorphic for
$s\in\C$ with poles of order at most $k+1$ in $(n+b-j)/m$.
\end{theorem}

\section{Regularized traces}
\label{s:RT}

\subsection{The Residue Trace (Noncommutative Residue)}

We have seen in Proposition \plref{p:no-trace} that the Hilbert
space trace $\Tr$ cannot be extended to all classical pseudodifferential
operators.

However, 
in his seminal papers \cite{Wod:LIS}, \cite{Wod:NRF}
\textsc{M. Wodzicki} was able to show
that, up to a constant,
the algebra $\CL^\bullet(M)$ has a unique trace which he called the
noncommutative residue; we prefer to call it residue trace. 
The residue trace was independently
discovered by \textsc{V. Guillemin} \cite{Gui:NPW} as a byproduct
of his axiomatic approach to the Weyl asymptotics.
In \cite{Les:NRP} the author generalized the residue trace to
the algebra $\CL^{\bullet,\bullet}(M,E)$. Strictly speaking
there is no residue trace on the full algebra $\CL^{\bullet,\bullet}(M,E)$.
Rather one has to restrict to operators with a given bound on
the $\log$ degree. 

In detail: let $A\in\CL^{a,k}(M,E)$ and let $P\in\CL^m(M,E)$ elliptic, non--negative
and invertible, cf. Subsection \ref{ss:heat-zeta}. Put

\begin{equation}\label{eq:def-NCR}
  \begin{split}
      \Res_k(&A,P)\\
      &:= m^{k+1} \Res_{k+1} \Tr(AP^{-s})|_{s=0}\\
    &= m^{k+1}(-1)^{k+1} (k+1)!\times \text{ coefficient of }\;
      \log^{k+1} t\text{ in the }\\
    &\quad \text{asymptotic  expansion of}\;\Tr(Ae^{-tP})\text{ as } t\to 0.\\
  \end{split}
\end{equation}

In \cite{Les:NRP} it was assumed in addition that the leading
symbol of $P$ is scalar. This assumption allows one to use Duhamel's
principle and to systematically
exploit the fact that the order of a commutator $[A,P]$ is at most
$\ord A+\ord P-1$. Using the resolvent approach
it was shown in \textsc{Grubb} \cite{Gru:RAT} that for defining $\Res_k$
and to derive its properties one does not need to assume that
$P$ has scalar leading symbol.

The main properties of $\Res_k$ can now be summarized as follows:

\begin{theorem}[Wodzicki--Guillemin; $log$--polyhomogeneous case \cite{Les:NRP}] 
\indent\par Let $A\in\CL^{a,k}(M,E)$ and let $P\in\CL^m(M,E)$
be elliptic, non--negative and invertible.
  
\textup{1.} $\Res_k(A,P)=:\Res_k(A)$ is independent of $P$, i.e. 
\[\Res_k:\CL^{\bullet,k}(M,E)\longrightarrow \C\] is a
  linear functional.

\textup{2.} If $A\in\CL^{a,k}(M,E), B\in\CL^{b,l}(M,E)$
  then $\Res_k([A,B])=0$. In particular, $\Res:=\Res_0$
  is a trace on $\CL^\bullet(M,E)$.

\textup{3.} For $A\in\CL^{a,k}(M,E)$ the $k$-th residue
  $\Res_k(A)$ vanishes if \[a\not\in -\dim M+\Z_+.\]

\textup{4.} In a local chart one puts 
 \begin{equation}
\go_k(A)(x)
=\frac{(k+1)!}{(2\pi)^n} \Big(\int_{|\xi|=1} \tr_{E_x}(a_{-n,k}(x,\xi)) |d\xi|
\Big) |dx|.
        \label{eq:residue-density}
\end{equation}
Then $\go_k(A)\in\Gamma^\infty(M,|\Omega|)$ is a density (in particular
independent of the choice of coordinates), which depends
  functorially on $A$. Moreover
  \begin{equation}
       \Res_k(A)=\int_M \go_k(A).
  \end{equation}

\textup{5.} If $M$ is connected and $n=\dim M>1$ then
$\Res_k$ induces an isomorphism
$\CL^{a,k}(M)/[\CL^{a,k}(M),\CL^{1,0}(M)]\longrightarrow\C$.
In particular, $\Res$ is up to scalar multiples the only
trace on $\CL^\bullet(M)$.
\end{theorem}

\commentary{\begin{remark}\TODO{no so clear}
  5. is usually stated only for scalar operators. However, the extension
  to operators in a vector bundle is straightforward.
\end{remark}
}

\begin{example}\label{ex:20090525}
1. Let $A$ be a classical pseudodifferential operator of order $-n=-\dim M$
which is assumed to be elliptic, non--negative and invertible. To calculate
the residue trace of $A$ we may use $P:=A\ii$. Thus 
\begin{equation}\label{eq:ML20090525-1}
\Res(A)=n \Res \Tr(A^{1+s})|_{s=0}=n \Res \zeta(A\ii;s)|_{s=1}>0,
\end{equation}
where $\zeta(A\ii;s)=\zeta(I,A\ii;s)$ is the $\zeta$--function of
the elliptic operator $A\ii$. The positivity follows from
Eq. \eqref{eq:residue-density}.

2. Let $\Delta$ be the Laplacian on a closed
  Riemannian manifold $(M,g)$. Then the heat expansion
  \eqref{eq:log-heat-expansion} (with $B=I$ and $P=\Delta$) simplifies: 
  since $\Delta$ is a differential operator there are no $\log$ terms
  and by a parity argument every other heat coefficient vanishes
  \cite{Gil:ITH}. Thus we have an asymptotic expansion
  \begin{equation}\label{eq:heat-expansion-delta}
    \Tr(e^{-t\Delta})\sim_{t\to 0} \sum_{j=0}^\infty a_j(\Delta)
    t^{(j-n)/2},\quad a_{2j+1}(\Delta)=0.
  \end{equation}
  The $a_j(\Delta)$ are enumerated such that \eqref{eq:heat-expansion-delta}
  is consistent with \eqref{eq:log-heat-expansion}. 
  The first few $a_j(\Delta)$ have been calculated\
  although the computational complexity increases drastically with $j$
  (cf. e.g. \cite{Gil:ITH}). One has
  \begin{equation}\label{eq:20081118-3}
    \begin{split}
      a_0(\Delta)&=c_n\vol(M)\\
      a_2(\Delta)&=c_n' \int_M \operatorname{scal}(M,g)d\vol.
    \end{split}
  \end{equation}
  The latter is known as the \emph{Einstein-Hilbert action} in the physics 
  literature. Therefore the following relation between the heat coefficients
  (and in particular the EH action) and the residue trace has received
  some attention from the physics community, e.g. \textsc{Kalau--Walze} \cite{KalWal:GNC}, 
  \textsc{Kastler} \cite{Kas:DOG}.
  We find for real $\ga$
  \begin{align}
    \Res(\Delta^\ga)&= 2 \lim_{s\to 0} s \Tr(\Delta^{\ga-s})\nonumber\\
            &= 2 \lim_{s\to 0} s\zeta(I,\Delta;s-\ga)\nonumber\\
	    &= 2\lim_{s\to 0} \frac{s}{\Gamma(s-\ga)}\int_0^1
	    t^{s-\ga-1}\bigl(\Tr(e^{-t\Delta})-\dim\ker\Delta\bigr)dt\label{eq:20081118-1}\\
	    &=2\sum_{j=0}^\infty \lim_{s\to 0}
	    \frac{a_j(\Delta)s}{\Gamma(s-\ga)(s-\ga+\frac{j-n}{2})}\label{eq:20081118-2}\\
	    &=  \begin{cases} \frac{2 a_j(\Delta)}{\Gamma(\frac{n-j}{2})},&
	                                                                \ga=\frac{j-n}{2}<0,\\
	                               0,&\text{otherwise.}
	                 	    \end{cases}\label{eq:20081118-4}
  \end{align}
  Here we have used that the $\zeta$--function of $\Delta$ has only
  simple poles (cf. Theorem \ref{t:zeta-meromorphic}). Furthermore,
  in \eqref{eq:20081118-1} we use that due to
  the exponential decay of $(\Tr(e^{-t\Delta})-\dim\ker\Delta)$ the function
  $s\mapsto \int_1^\infty t^{s-\ga-1}(\Tr(e^{-t\Delta})-\dim\ker\Delta)dt$ is entire and
  hence does not contribute to the residue at $s=0$. Furthermore,
  note that the sum in \eqref{eq:20081118-2} is finite.

  In view of \eqref{eq:20081118-3} we have the following special cases of
  \eqref{eq:20081118-4}:
 \begin{align}
      \Res(\Delta^{-n/2})&=\frac{2 a_0(\Delta)}{\Gamma(\frac n2)}=c_n\vol(M)\label{eq:20081118-5},\\
      \Res(\Delta^{1-n/2})&=c_n' \operatorname{EH}(M,g),\label{eq:20081118-6}
  \end{align}
  where $\operatorname{EH}$ denotes the above mentioned Einstein-Hilbert
  action. It is formula \eqref{eq:20081118-6} which caused physicists to become
  enthusiastic about this business. Needless to say, the calculation we
  present here goes through for any Dirac Laplacian. One only has to replace
  the scalar curvature in \eqref{eq:20081118-3} by the second local heat
  coefficient, which can be calculated for any Dirac Laplacian.
  
  We wanted to show that the relation
  between the heat asymptotic and the poles of the $\zeta$--function,
  which is an easy consequence of the Mellin transform, leads to a
  straightforward proof of \eqref{eq:20081118-6}. There also exist
  ``hard'' proofs of this fact which check that the \emph{local}
  Einstein-Hilbert action coincides with the residue density of the operator
  $\Delta^{1-n/2}$ \cite{KalWal:GNC},\cite{Kas:DOG}.
\end{example}

\subsection{Connes' Trace Theorem}

The famous trace Theorem of Connes gives a relation between the Dixmier
trace and the Wodzicki--Guille\-min residue trace for pseudodifferential
operators of order minus $\dim M$. It was extended by \textsc{Carey} et. al.
\cite{Caretal:SFD}, \cite{Caretal:DTA} to the von Neumann algebra setting.

\begin{theorem}[Connes' Trace Theorem {\cite{Con:AFN}}]\label{t:connes-trace} 
Let $M$ be a closed manifold of dimension $n$ and let $E$ be a smooth vector bundle
over $M$. Furthermore let $P\in\CL^{-n}(M,E)$ be a pseudodifferential operator
of order $-n$. Then $P\in\cL^{(1,\infty)}(L^2(M,E))$ and for any $\go$ satisfying
the assumptions of the previous Proposition one has
\begin{equation}
     \Tr_\go(P)=\frac 1n \Res P.
\end{equation}
\end{theorem}

We give a sketch of the proof of Connes' Theorem using a Tauberian argument. This was mentioned
without proof in \cite[Prop. 4.2.$\beta$.4]{Con:NG} and has been elaborated in various ways
by many authors. The argument we present here is an adaption of an argument in
\cite{Caretal:SFD} to the type I case.

Let us mention the following simple version of Ikehara's Tauberian Theorem:

\begin{theorem}[{\cite[Sec. II.14]{Shu:POST}}] \label{t:Ikehara} Let $F:[1,\infty)\to \R$ be an increasing function such that
  \begin{enumerate}
    \item[\textup{(1)}] $\zeta_F(s)=\int_1^\infty \gl^{-s} dF(\gl)$ is analytic for $\Re s>1$,
    \item[\textup{(2)}] $\lim\limits_{s\to 1+} (s-1)\zeta_F(s)=L$. 
  \end{enumerate}
Then 
\begin{equation}
   \lim\limits_{\gl\to\infty} \frac{F(\gl)}{\gl}=L.
  \label{eq:20081114-4}
\end{equation}
\end{theorem}

\begin{cor}\label{t:20081114-11}
  Let $F:[1,\infty)\to \R$ be an increasing function such that
  $\int_1^\infty e^{-t\gl} dF(\gl)=\frac{L}{t}+O(t^{\eps-1}), t\to 0+,$
  for some $\eps>0$. Then Ikehara's Theorem applies to $F$ and \eqref{eq:20081114-4}
  holds.
\end{cor}
\begin{proof} The $\zeta$--function of $F$ satisfies
  \begin{equation}
    \begin{split}
      \zeta_F(s)&=  \int_1^\infty \gl^{-s}dF(\gl)\\
              &=  \int_1^\infty \frac{1}{\Gamma(s)} \int_0^\infty t^{s-1}e^{-t\gl}dt\; dF(\gl)\\
	      &= \int_0^1 \frac{t^{s-1}}{\Gamma(s)} \int_1^\infty e^{-t\gl}dF(\gl)\; dt+\text{ holomorphic near } s=1\\
              &\sim \frac{1}{\Gamma(s)}\frac{L}{s-1}\text{ near } s=1.\qedhere
    \end{split}
    \label{eq:20081114-5}
  \end{equation}
\end{proof}
\begin{proof}[Proof of Connes' Trace Theorem]
  Each $P\in\CL^{-n}(M,E)$ is a linear combination of at most $4$ non--negative operators:
  to see this we first write $P=\frac 12(P+P^*)+\frac{1}{2i}(P-P^*)$ as a linear
  combination of two self--adjoint operators. So consider a self--adjoint $P=P^*$. We
  choose an elliptic operator $Q\in \CL^{-n}(M,E)$ with $Q>0$ and positive definite leading symbol.
  Since we are on a compact manifold it then follows that $c\cdot Q-P\ge 0$ for $c$ large enough.
  Hence $P=c\cdot Q-(c\cdot Q-P)$ is the desired decomposition of $P$ as a difference of non--negative
  operators.

  So it suffices to prove the claim for a non--negative operator $P$. Then $P+\eps Q$ is
  elliptic and invertible for each $\eps>0$. By an approximation argument we are ultimately
  left with the problem of proving the claim for an 
\emph{elliptic} positive operator $P\in\CL^{-n}(M,E)$.

  Let $\mu_1\ge \mu_2\ge \mu_3\ge \dots>0$ be the eigenvalues of $P$ counted with multiplicity.
  We consider the counting function
  \begin{equation}
    F(\gl)=\#\bigsetdef{j\in\N}{\mu_j^{-1}\le\gl}.
    \label{eq:20081114-6}
  \end{equation}
  The associated $\zeta$--function
  \begin{equation}
    \zeta_F(s)=\int_1^\infty \gl^{-s}dF(\gl)=\Tr(P^{s})-\sum_{\mu_j>1}\mu_j^{s}
    \label{eq:20081114-7}
  \end{equation}
  is, up to the entire function $\sum\limits_{\mu_j>1}\mu_j^{s}$, the $\zeta$--function of the elliptic operator $P^{-1}$. Thus by 
  Theorem \ref{t:zeta-meromorphic} the function 
  $\zeta_F$ is holomorphic for $\Re s>1$ and it has a meromorphic extension to the complex plane,
and $1$ is a simple pole with 
  \begin{equation}
    \lim_{s\to 1} (s-1)\zeta_F(s)=\frac 1n \Res(P)\not=0,
    \label{eq:20081114-8}
  \end{equation}
cf. Example \plref{ex:20090525} 1.
Thus Ikehara's Theorem \plref{t:Ikehara} applies to $F$ and hence
\begin{equation}\label{eq:20090511-5}
       \lim_{\gl\to\infty} \frac{F(\gl)}{\gl}=\frac 1n \Res(P).
\end{equation} 

\emph{Claim:} 
\begin{equation}\label{eq:claim}
\lim\limits_{j\to\infty} j \mu_j=\frac 1n \Res(P)=:L.
\end{equation}
To see this let $\eps>0$ be given. Then there exists a $\gl_0$ such
that for $\gl\ge\gl_0$
\begin{equation}
     1-\eps \le \frac{F(\gl)}{\gl L} \le 1+\eps.
\end{equation}
Thus
\begin{equation}
\exists_{\gl_0}\forall_{\gl\ge\gl_0} \quad (1-\eps) \gl L\le 
\#\bigsetdef{j\in\N}{\mu_j^{-1}\le\gl}\le (1+\eps)\gl L.
\end{equation}
Hence for $j\ge (1+\eps)\gl L$ we have $\mu_j^{-1}\ge\gl$
and for $j\le (1-\eps)\gl L$ we have $\mu_j^{-1}\le \gl$.
For a given fixed $j_0$ large enough we therefore infer
\begin{equation}\label{eq:20090512-6}
    (1-\eps)L\le j \mu_j\le (1+\eps) L,\quad j\ge j_0,
\end{equation}
proving the Claim. 

Now consider
  \begin{equation}
    \begin{split}
      \gb(u)=\int_1^{e^u} \gl^{-1}dF(\gl) =\sum_{\mu_j\ge e^{-u}} \mu_j.
    \end{split}
    \label{eq:20081114-9}
  \end{equation}
  We check that Ikehara's Tauberian Theorem applies to $\gb$:
  \begin{equation}
    \begin{split}
      \int_1^\infty &e^{-s\gl}d\gb(\gl)=\int_1^\infty e^{-(s+1)\gl}dF(e^\gl)\\
      &=\int_e^\infty x^{-s-1}dF(x)=\zeta_F(1+s)\\
      &=\frac{\Res(P)}{n s} +O(1),\quad s\to 0.
\end{split}
    \label{eq:20081114-10}
  \end{equation}
  Thus Corollary \plref{t:20081114-11} implies 
  \begin{equation}
    \frac 1u \sum_{\mu_j\ge e^{-u}} \mu_j=\frac{\gb(u)}{u}\xrightarrow{u\to\infty} \frac 1n \Res(P). 
    \label{eq:20081114-12}
  \end{equation}

To infer Connes' Trace Theorem from \eqref{eq:20081114-12} we choose $j_0$
such that \eqref{eq:20090512-6} holds for $\eps=1/2$ and $j\ge j_0$.
Then put for $N$ large enough $u_N:=\log\frac{N}{(1-\eps)L}$. Hence
we have $\mu_j\ge \mu_N\ge e^{-u_N}$ for $1\le j\le N$ and thus
\begin{equation}
\begin{split}
   \frac{1}{\log(N+1)}\sum_{j=1}^N\mu_j &\le
\frac{1}{\log(N+1)}\sum_{\mu_j\ge \exp(-u_N)} \mu_j\\
&=\frac{u_N}{\log{N+1}} \frac{1}{u_N} \sum_{\mu_j\ge \exp(-u_N)} \mu_j\\
&\longrightarrow L,\quad\text{for }N\to \infty,
\end{split}
\end{equation}
by \eqref{eq:20081114-12} and since $u_N/\log(N+1)\to 1$.
This proves
\begin{equation}
   \limsup_{N\to\infty} \frac{1}{\log(N+1)}\sum_{j=1}^N \mu_j\le L=\frac
1n\Res(P).
\end{equation}
Arguing with $u_N=\log\frac{N}{(1+\eps)L}$ instead of $u_N=\log\frac{N}{(1-\eps)L}$
one shows 
\begin{equation}
   \liminf_{N\to\infty} \frac{1}{\log(N+1)}\sum_{j=1}^N \mu_j\ge L=\frac
1n\Res(P),
\end{equation}
and Connes' Trace Theorem is proved.
\end{proof}

The attentive reader might have noticed that we did not use the full
strength of the Claim \eqref{eq:claim}. We only used that there
exist positive constants $c_1,c_2$ such that
$c_1\le j\mu_j\le c_2$ for $j\ge j_0$.


\subsection{Parametric case: The symbol valued trace}

In contrast to Proposition \plref{p:no-trace} the situation is entirely
different for the algebra of parametric pseudodifferential
operators. 

Fix a compact smooth manifold $M$ without boundary of dimension $n$.  
Denote the coordinates in $\R^p$ by $\mu_1,\ldots,\mu_p$ and
let $\polyn$ be the algebra of polynomials
in $\mu_1,\ldots,\mu_p$. By a slight abuse of notation we denote
by $\mu_j$ also the operator of multiplication by the
$j$-th coordinate function. Then we have maps
\begin{equation}\label{eq:20081120-5}
\begin{split}
   &\partial_j:\CL^m(M,E;\R^p)\rightarrow \CL^{m-1}(M,E;\R^p),\\
   &\mu_j:\CL^m(M,E;\R^p)\rightarrow  \CL^{m+1}(M,E;\R^p).
\end{split}
\end{equation}
Also $\partial_j$ and $\mu_j$ act naturally on the parametric
symbols over the one--point space 
$\CS^{\bullet,\bullet} (\R^p):=\CS^{\bullet,\bullet}(\{\textup{pt}\};\R^p)$
and on polynomials $\polyn$. Thus they act on the quotient
$\CS^{\bullet,\bullet} (\R^p)/\polyn$.
After these preparations we can summarize one of the main results
of \cite{LesPfl:TAP}. 

Let $E$ be a smooth vector bundle on $M$ and consider $A\in\CL^m(M,E;\R^p)$
with $m+n< 0$. Then for $\mu\in\R^p$
the operator $A(\mu)$ is trace class; hence we may define the function
$\TR(A):\mu\mapsto \Tr(A(\mu))$. The map $\TR$ is obviously tracial, 
i.e.~$\TR(AB)=\TR(BA)$, and commutes with $\partial_j$ and $\mu_j$. 
In fact, the following theorem holds.

\begin{theorem}\textup{\cite[Theorems 2.2, 4.6 and Lemma 5.1]{LesPfl:TAP}}
  \label{t:Lesch-Pflaum}
There is a unique linear extension
\[\TR:\CL^\bullet (M,E;\R^p)\rightarrow \CS^{\bullet,\bullet}(\R^p)/\polyn\]
of $\TR$ to operators of all orders such that
\begin{enumerate}
\item \label{thm21.1}$\TR(AB)=\TR(BA)$, i.e. $\TR$ is tracial.
\item \label{thm21.2}$\TR(\partial_j A)=\partial_j \TR(A)$ for $j=1,\dots,p$.
\end{enumerate}
This unique extension $\TR$ satisfies furthermore:
\begin{enumerate}
\setcounter{enumi}{2}
\item \label{thm21.3}$\TR(\mu_j A)=\mu_j\TR(A)$ for $j=1,\dots,p$.
\item \label{thm21.4}$\TR(\CL^m(M,E;\R^p))\subset \CS^{m+p,1}(\R^p)/\polyn$.
\end{enumerate}
\end{theorem}
This Theorem is an example where functions with $\log$--poly\-homo\-ge\-neous
expansions occur naturally. Note that although an operator
$A\in\CL^m(M,E;\R^p)$ has a homogeneous symbol expansion without
$\log$ terms the trace function $\TR(A)$ is $\log$--polyhomogeneous.

\begin{proof}[Sketch of Proof]
\commentary{We briefly present two arguments which help explain why this
theorem is true.

\subsubsection*{1.~Taylor expansion} 
 Given $A \in \CL^m (M,E;\R^p)$. Since differentiation by the parameter
reduces the order of the operator, the Taylor expansion around $0$
yields for $\mu\in\R^p$ (cf.~\cite[Prop.~4.9]{LesPfl:TAP})
 \begin{equation} \label{LesPfl:G1-3.7}
   A(\mu) - \sum_{|\alpha| \leq N-1} \, 
   \frac{(\partial_\mu^\alpha A)(0)}{\alpha !} \, 
   \mu^\alpha \in \CL^{m-N} (M,E).
 \end{equation}
Hence, if $N$ is so large that $m-N+n < 0$, then the difference
\eqref{LesPfl:G1-3.7} is trace class and we put
 \begin{equation}
 \label{LesPfl:G1-3.8}\begin{split}
   \TR(A)(\mu) 
    := \tr &\Big( A(\mu) - \sum_{|\alpha| \leq N-1}
   \frac{(\partial_\mu^\alpha A) (0)}{\alpha !} \, \mu^\alpha
   \Big) \; \mod {\polyn}.
 \end{split}
\end{equation}
Since we mod out by polynomials, the result is in fact independent
of $N$. This defines $\TR$ for operators of all order and the properties
(1)--(3) are straightforward to verify.
However,  \eqref{LesPfl:G1-3.7} does not give any asymptotic 
information and hence does not justify the fact that $\TR$ takes
values in $\PS$.
}
The main observation for the proof is that differentiating by the 
parameter \eqref{eq:20081120-5}
lowers the degree and hence differentiating often enough we obtain 
a parametric family of trace class operators:

Given $A\in \CL^m (M,E;\R^p)$, then 
$\partial^\alpha A\in \CL^{m-|\alpha|}(M,E,\R^p)$ is of
trace class if $m-|\alpha|+ \dim M <0$. Now integrate the function 
$\TR(\partial^\alpha A)(\mu)$
back. Since we mod out polynomials this procedure is independent
of $\alpha$ and the choice of anti--derivatives. This integration procedure
also explains the possible occurrence of $\log$ terms in the
asymptotic expansion and hence why $\TR$ ultimately takes values in
$\CS^{\bullet,\bullet}(\R^p)$. 
For details, see \cite[Sec.~4]{LesPfl:TAP}. 
\end{proof}

$\TR$ is not a trace in the usual sense since it maps into
a quotient space of the space of parametric symbols over a point. 
However, composing any linear functional on $\CS^{\bullet,\bullet}(\R^p)/\polyn$
with $\TR$ yields a trace on $\CL^\bullet (M,E;\R^p)$. 
A very natural choice for such a trace is the Hadamard partie finie
integral $\reginttext$ introduced in Subsection \plref{ss:partie-finie}.
Let us first note that for a polynomial $P(\mu)\in\C[\mu_1,\dots,\mu_p]$
of degree $r$ the function
\begin{equation}
\int_{|\mu|\le R} P(\mu)d\mu=\sum_{j=p}^{p+r} a_j R^j
\end{equation}
is a polynomial of degree $p+r$ without constant term. In particular
\begin{equation}
\regint_{\R^p} P(\mu)d\mu=0
\end{equation}
and hence $\regint_{\R^p}$ induces a linear functional on 
the quotient space\\  $\CS^{\bullet,\bullet}(\R^p)/\polyn$.

Thus putting for $A\in\CL^\bullet(M,E;\R^p)$
\begin{equation}\label{eq:parametric-trace}
    \overline{\TR}(A):=\regint_{\R^p} \TR(A)(\mu)d\mu
\end{equation}
we obtain a trace $\overline{\TR}$ on $\CL^\bullet(M,E;\R^p)$
which extends the natural trace on operators of order
$<-\dim M -p$
\begin{equation}\label{eq:def-fTR}
\bigl(\int\Tr\bigr)(A):=\int_{\R^p}\Tr(A(\mu))d\mu.
\end{equation}

However, since $\reginttext$ is not closed
on $\CS^{\bullet,\bullet}(\R^p)$ (Prop. \plref{S2-4.4}),
$\overline{\TR}$ is not closed on $\CL^\bullet(M,E;\R)$.
Therefore we obtain derived traces
\begin{equation}
  \partial_j\overline{\TR}(A):=\widetilde\TR_j(A)
  :=\regint_{\R^p}\TR(\partial_j A)(\mu)d\mu.
\end{equation}


The relation between $\fTR$ and $\widetilde{\TR}_j$ can be explained
more elegantly in terms of differential forms on $\R^p$ with coefficients
in $\CL^\infty (M,E;\R^p)$ (see \textsc{Lesch, Moscovici} and
\textsc{Pflaum} \cite{LesMosPfl:RPC}). Let 
$\Lambda^\bullet:= \Lambda^\bullet (\R^p)^*=\C[d\mu_1,\ldots,d\mu_p]$ 
be the exterior algebra of the vector space $(\R^p)^*$ and put
\begin{equation}\label{ML-G2.8}
    \Omega_p:=\CL^\infty (M,E;\R^p)\otimes \Lambda^\bullet .
\end{equation}
Then, $\Omega_p$ consists of pseudodifferential operator-valued
differential forms, the coefficients of $d\mu_I$ being
elements of $\CL^\infty (M,E;\R^p)$.

For a $p$-form $A(\mu)d\mu_1\wedge\ldots\wedge d\mu_p$
we define the \emph{regularized trace} by
\begin{equation}\label{ML-G2.9}
     \fTR(A(\mu)d\mu_1\wedge\ldots\wedge d\mu_p)
   := \regint_{\R^p} \TR(A)(\mu)d\mu_1\wedge\ldots\wedge d\mu_p.
\end{equation}
On forms of degree less than $p$ the regularized trace
is defined to be $0$. $\fTR$ is a \emph{graded trace} on the
differential algebra $(\Omega_p,\, d)$. In general, $\fTR$ is
not closed. However, its boundary, 
$$
  \lTR:= d\fTR := \fTR\circ d \, ,
$$ 
called the \emph{formal trace},
is a closed graded trace of degree $p-1$. It is
shown in \cite[Prop.~5.8]{LesPfl:TAP}, \cite[Prop.~6]{Mel:EIF}
that $\lTR$ is \emph{symbolic}, i.e.~it descends to a well-defined
closed graded trace of degree $p-1$ on
\begin{equation}\label{ML-G2.10}
  \partial \Omega_p:= \CL^\infty (M,E;\R^p)/\CL^{-\infty}(M,E;\R^p)
  \otimes\Lambda^\bullet.
\end{equation}

The properties of the formal trace $\lTR$ resemble those of the
residue trace.

Denoting by $r$ the quotient map $\Omega_p\to\partial \Omega_p$
we see that Stokes' formula with `boundary' 
\begin{equation}\label{ML-G2.11}
   \fTR(d\omega)=\lTR(r\omega)
 \end{equation}
now holds by construction for any $\omega\in\Omega$.

Finally we mention an interesting linear form on 
$\CS^{\bullet,\bullet}(\R^p)/\polyn$ 
in the spirit of the residue trace. Let
\begin{equation}
\Omega^r\CS^{\bullet,\bullet}(\R^p)=\CS^{\bullet,\bullet}(\R^p)\otimes\Lambda^\bullet
\end{equation}
be the $r$--forms on $\R^p$ with coefficients in
$\CS^{\bullet,\bullet}(\R^p)$. We extend the notion of homogeneous
functions to differential forms in the obvious way. 
If $\go=f d\mu_{i_1}\wedge\dots\wedge d\mu_{i_r}$ is a form of degree $r$
and $f\in \CS^{a,k}(\R^p)$ then we define the \emph{total degree} of 
$\go$ to be $r+a$. The exterior derivative preserves the total degree
and each $\go\in\Omega^\bullet\CS^{\bullet,\bullet}(\R^p)$
of total degree $a$
has an asymptotic expansion
\begin{equation}
     \go\sim \sum_{j=0}^\infty \go_{a-j}
\end{equation}
where $\go_{a-j}$ are forms of total degree $a-j$ which are $\log$--polyhomogeneous
in the sense of \eqref{ML-G2.2}, see \eqref{eq:classical}.
More concretely, if $f\in\CS^{a,k}(\R^p)$ then for
$\go=f\,d\mu_1\wedge\dots d\mu_r$ we have
\begin{equation}
     \go_{a+r-j}=f_{a-j}.
\end{equation}
Accordingly we define $\go_{a+r-j,l}:=f_{a-j,l}$.

Finally let $X=\sum_{j=1}^p \mu_j\frac{\pl}{\pl \mu_j}$ be the Liouville
vector field on $\R^p$.

After these preparations we put for $\go=fd\mu_1\wedge\dots\wedge d\mu_p\in
\Omega^p\CS^{\bullet,\bullet}(\R^p)$
\begin{equation}
    \res(\go):=\frac{1}{(2\pi)^p}\int_{S^{p-1}} i_X(\go_0)=
          \frac{1}{(2\pi)^p}\int_{S^{p-1}} f_{-p,0} d\vol_S.
\end{equation}    
On forms of degree $<p$ we put $\res(\go)=0$.

\begin{prop}\label{p:Stokes-property} 
If $f\in\C[\mu_1,\dots,\mu_p]$ is a polynomial then
\[\res(fd\mu_1\wedge\dots\wedge d\mu_p)=0.\]

If $\go\in\Omega^\bullet\CS^{a,0}(\R^p)$ then $\res(d\go)=0$.
\end{prop}
\noindent The second statement is due to \textsc{Manchon, Maeda} and \textsc{Paycha}
\cite{Manetal:SFC}.
\begin{proof}
For $f\in\C[\mu_1,\ldots,\mu_p]$ the component of 
homogeneity degree $0$ of
$fd\mu_1\wedge\dots\wedge d\mu_p$ is obviously $0$.

Using Cartan's identity we have
\begin{equation}
\begin{split}
     \res(d\go)&=\int_{S^{p-1}} i_X(d \go_0)=\int_{S^{p-1}} (i_Xd+d i_X)(\go_0)\\
	       &=\int_{S^{p-1}} \mathcal{L}_X \go_0 =0,
 \end{split}
\end{equation}
since the Lie derivative of a form of homogeneity degree $0$
with respect to the Liouville vector field $X$ is $0$.
\end{proof}

Composing the $\res$ functional with $\TR$ we obtain another trace
on the algebra $\CL^\bullet(M,E;\R^p)$ which despite 
of the previous Proposition is not closed. The point
here is that the range of $\TR$ is not contained
in $\CS^\bullet(\R^p)$ but rather in $\CS^{\bullet,1}(\R^p)$.

The significance of this functional and its relation to the noncommutative
residue is still to be clarified.


  
\newcommand{\rrr}{\!\!\upharpoonright\!} 
\newcommand{\sa}{\textup{sa}}
\newcommand{\half}{{1/2}}
\newcommand{\comp}{\operatorname{comp}}
\newcommand{\dstr}{d_{\textrm{str}}
}

\newcommand{\checknote}{\marginpar{check}}

\newcommand{\cl}{\textup{cl}}

\newcommand{\halfline}{[0,\infty)}
\newcommand{\frechet}{Fr{\'e}chet}

\section{Differential forms whose coefficients are symbol functions}
\label{s:DFC}

Proposition \plref{p:Stokes-property} says
that the $\res$ functional on $\Omega^\bullet\CS^{\bullet}(\R^n)$
descends to a linear functional on the $n$--th de Rham cohomology
of differential forms with coefficients in $\CS^{\bullet}(\R^n)$.
In \textsc{Paycha} \cite{Pay:NRC} it is shown that the space
of linear functionals on $\CS^{\bullet}(\R^n)$ having the Stokes property
is one--dimensional. From this statement in fact the uniqueness
of the residue trace can be derived. Translated into our terminology this
means that the dual of the $n$--th de Rham cohomology group of $\R^n$ with
coefficients in $\CS^{\bullet}(\R^n)$ is spanned by $\res$.
In particular the $n$-th de Rham cohomology group of $\R^n$ 
with coefficients in $\CS^{\bullet}(\R^n)$
is one--dimensional. In \cite{Pay:NRC} it is shown furthermore
that the uniqueness statement for linear functionals having the
Stokes property is basically equivalent to the uniqueness
statement for the residue trace.

We take up this theme and study in a rather general setting
the de Rham cohomology of differential forms whose coefficients 
are symbol functions. The
results announced here are inspired by 
\cite{Pay:NRC} but are more general. 
We pursue here an axiomatic approach. Details will appear elsewhere. 


\subsection{Differential forms with prescribed asymptotics}

\begin{dfn}\label{def-1.1} Let $\cA\subset C^\infty{\halfline}$ be a 
Fr{\'e}chet space with the following properties.

\begin{enumerate}
\item $\cinfz{\halfline}\subset \cA\subset \cinf{\halfline}$ are continuous embeddings. $\cinf{\halfline}$ carries
the usual Fr{\'e}chet topology of uniform convergence of all derivatives on compact sets and $\cinfz{\R}$
has the standard LF-space topology as inductive limit of the \frechet\ spaces $\bigsetdef{f\in\cinf{\halfline}}{
\supp f\subset [0,N]}$, $N\in\N$.

We denote by $\cA_0=\bigsetdef{f\in\cA}{\supp f\subset (0,\infty)}.$
\item The derivative $\pl:=\frac{d}{dx}$ maps $\cA$ into $\cA$.
\item There is a non--trivial linear functional $\reginttext:\cA\to \C$ with
the following properties:
\begin{enumerate}
\item The restriction of $\reginttext$ to $\cinfz{\halfline}$ is a multiple of the integral
$\int_0^\infty$. That is, there is a $\gl\in\C$ such that for $f\in\cinfz{\halfline}$
we have $\reginttext f=\gl\int_0^\infty f(x) dx$.
\item $\reginttext$ is \emph{closed} on $\cA_0$. That is, for $f\in\cA_0$ we have $\reginttext f=0$.
\item If $f\in \cA_0$ and $\reginttext f=0$ then the function $F:=\int_0^\bullet f\in\cA$. 
\end{enumerate}
\end{enumerate}
\end{dfn}
\begin{remark}
 It follows from (1) that 
if $\chi\in\cinf{\halfline}$ with $\chi(x)=1, x\ge x_0$ and $f\in\cA$ then
$\chi f\in \cA$ because $(1-\chi)f\in\cinfz{\halfline}\subset\cA$.

2. Since $\cA$ is \frechet\, it follows from (1) and (2) and the Closed Graph Theorem
that $\frac{d}{dx}:\cA\to\cA$ is continuous.

3. If $\gl$ in (3a) is nonzero we can renormalize $\reginttext$ such that $\gl=1$. Thus
we are left with two major cases: $\gl=1$ and $\gl=0$. In the first case $\reginttext$
is a regularization of the ordinary integral while in the second case $\reginttext$
is an analogue of the residue trace. This will be explained below in the examples. 
\end{remark}

\begin{example}
1. The Schwartz space $\cS(\R)$, $\reginttext=\int$.

2. Let $\CS^a(\halfline)$, $a\in\halfline$ be the classical symbols of order $a$. 
This space carries a natural \frechet\ topology.
If $a\not\in\{-1,0,1,\dots\}$ then let $\reginttext$ be the regularized
integral in the partie finie sense described in Subsection \plref{ss:partie-finie}.
This integral is continuous with respect to the \frechet\ topology
on $\CS^a(\halfline).$

If $a\in\{-1,0,1,\dots\}$ then let $\reginttext$ be the 
residue integral (cf. \eqref{eq:residue-density}),
i.e. if 
\begin{equation}
f(x)\sim_{x\to \infty}\sum_{j=0}^\infty f_{a-j} x^{a-j}
\end{equation}
then
\begin{equation}
  \regint f := f_{-1}.
\end{equation}

One can vary this example. With some care one can also deal with
$\log$--polyhomogeneous symbols. Moreover, there are classes of symbols
of integral order where the regularized integral has the Stokes property
\cite{Pay:NRC}. These ``odd class symbols'' also fit into the present
framework.
\end{example}

From now on $\cA$ will always denote a \frechet\ space as in
Def. \plref{def-1.1}.

Starting from $\cA$ we can construct associated spaces of functions
on $\R^n$ respectively on cones over a manifold.

Let $M$ be an oriented compact manifold. By $\cA_0(\halfline\times M)$
we denote the space of functions $f\in\cinf{\halfline\times M}$
such that
\begin{itemize}
\item There is an $\eps>0$ such that $f(r,p)=0$ for $r<\eps, p\in M$.
\item For fixed $p\in M$ we have $f(\cdot,p)\in\cA$.
\end{itemize}
Note that for $f\in\cA_0(\halfline\times M)$ the map
$M\to \cA, p\mapsto f(\cdot,p)$ is smooth. This follows
from the Closed Graph Theorem.
\details{Sketch: Since $\cA\hookrightarrow \cinf{\halfline}$
is continuously embedded and $M\ni p\mapsto f(\cdot,p)\in\cinf{\halfline}$
is smooth the claim follows.}

As a consequence we have a continuous integration along the fiber
\begin{equation}\label{eq:int-along-fiber}
  \regint_{(\halfline\times M)/M}:\cA_0(\halfline\times M)\longrightarrow
\cinf{M}, \quad f\mapsto \regint f(\cdot,p).
\end{equation}
We put
\begin{equation}
\cA_0(\R^n)=\bigsetdef{\pi^*f}{f\in \cA_0(\halfline\times S^{n-1}},
\end{equation}
where $\pi:\R^n\setminus\{0\}\longrightarrow \halfline\times S^{n-1}, x\mapsto (\|x\|, x/\|x\|)$
is the polar coordinate diffeomorphism.

Furthermore we put $\cA(\R^n):= \cinfz{\R^n}+\cA_0(\R^n)$. $\cA_0(\R^n)$ carries a natural LF-topology
while $\cA(\R^n)$ carries a natural \frechet\ topology.

\begin{remark} Composing the integral \eqref{eq:int-along-fiber} with an
  integral over $M$ yields a natural integral on $\cA_0(\halfline)\times
  M)$. In the case of $M=S^{n-1}$ and the standard integral on $S^{n-1}$
  this integral even extends to an integral
  on $\cA(\R^n)$ which has the Stokes property. 
  If $\cA=\CS^a([0,\infty))$ the so constructed integral on $\cA(\R^n)$
  is the Hadamard regularized
  integral if $a\not\in\{-1,0,1,\dots\}$ and the residue integral
  if $a\in\{-1,0,1,\dots\}$.
  Thus our approach allows us to discuss these two, a priori rather different,
  regularized integrals within one common framework.
\end{remark}

Finally we denote by $\Omega^k\cA_0(\halfline\times M)$ the space of differential forms
whose coefficients are locally in $\cA_0(\halfline\times U)$ for any chart $U\subset M$.
A more global description in terms of projective tensor products is also possible:
\begin{equation}
\cA_0(\halfline\times M)= \cA_0\otimes_\pi \cinf{M}, 
\end{equation}
respectively
\begin{equation}
\Omega^\bullet \cA_0(\halfline\times M)= (\cA_0\oplus \cA_0 dr)\otimes_\pi \Omega^\bullet(M). 
\end{equation}

By Def. \plref{def-1.1}, (2) the exterior derivative maps $\Omega^k\cA_{(0)}(X)$ to $\Omega^{k+1} \cA_{(0)}(X)$
for $X=\halfline\times M$, respectively $X=\R^n$. The corresponding cohomology groups are denoted by
$H^k \Omega^\bullet\cA_{(0)}(X)$. Our goal is to calculate these cohomology groups.

\begin{dfn} We call the $\cA$ of \emph{type I} if $\gl$ in Def.
  \plref{def-1.1} (3a) is $1$ and of \emph{type II} if $\gl$ is $0$.
\end{dfn}

\begin{lemma}\label{l:20081121-1} $\cA$ is of type II if and only if the constant function $1$
  is in $\cA$. Moreover we have for $k=0,1$
  \begin{equation}
      H^k\cA(\halfline)\simeq
      \begin{cases}0&, \text{ if } \cA \text{ is of type I,}\\
	           \C&, \text{ if } \cA \text{ is of type II.}
      \end{cases}
    \end{equation}
    $H^k\cA(\halfline)$ (obviously) vanishes for $k\ge 2$.
    Furthermore $\reginttext$ induces an isomorphism
    $H^1\cA_0(\halfline)\simeq \C$.
\end{lemma}

\subsection{Integration along the fiber and statement of the main result}
\label{ss:int-fiber}

\subsubsection{Integration along the fiber}\label{sec-2.1}

The integration \eqref{eq:int-along-fiber} extends to an integration along the
fiber of differential forms as follows (cf. \cite{BotTu:DFA}):

A $k$--form $\go\in\Omega^k\cA_0(\halfline\times M)$ is, locally on $M$, a sum
of differential forms of the form
\begin{equation}\label{eq:special-forms}
    \go= f_1(r,p)\pi^*\eta_1+f_2(r,p) \pi^*\eta_2\wedge dr
\end{equation}
with $f_j\in\cA_0(\halfline\times M), \eta_1\in \Omega^k(M), \eta_2\in\Omega^{k-1}(M)$.
For such forms we put
\begin{equation}
\pi_*\go:= \Bigl( \regint_{(\halfline\times M)/M} f_2\Bigr)\pi^*\eta_2.
\end{equation}

\begin{lemma} $\pi_*$ extends to a well--defined homomorphism
\[\Omega^k \cA_0(\halfline\times M)\longrightarrow 
\Omega^{k-1}\cA_0(\halfline\times M).\]
Furthermore, $\pi_*$ commutes with exterior differentiation, i.e.
\[d_M\circ \pi_* = \pi_*\circ d_{\R_+\times M}.\]
\end{lemma}
For the proof of this Lemma the closedness of $\reginttext$ is crucial.

\subsubsection{Statement of the main result}\label{sec-2.2}

We are now able to state our main result:

\begin{theorem} \label{thm:1}
  \textup{Type I:} If $\cA$ is of type I then the natural
  inclusion $\Omega_c^\bullet(\R^n)\hookrightarrow \Omega^\bullet\cA(\R^n)$
  of compactly supported forms induces an isomorphism in cohomology.
  
  \textup{Type II:} If $\cA$ is of type II then
  \begin{equation}
    H^k\cA(\R^n)\simeq\begin{cases} \C,& k=0,1,n,\\
                                    0,&\text{otherwise.}
         		  \end{cases}
  \end{equation}
  In both cases $\reginttext$ induces an isomorphism 
$H^n\cA(\R^n)\longrightarrow\C$.
\end{theorem}
\begin{remark}
1. The groups $H^k\cA(\R^n)$ can be described more explicitly.
Namely, the natural inclusion $\Omega^\bullet\cA_0(\R^n)\hookrightarrow
\Omega^\bullet\cA(\R^n)$ induces isomorphisms
\[H^k\cA_0(\R^n)\longrightarrow H^k\cA(\R^n)\]
for $k\ge 1$. Furthermore, integration along the fiber
induces isomorphisms 
\[\pi_*: H^k\cA_0(\R^n)\longrightarrow H^{k-1}(S^{n-1}), \quad \text{for }
k\ge 1.\]

Thus there is a natural extension of integration along the fiber
to closed forms $\pi_*:\Omega_{\cl}^k\cA(\R^n)\to\Omega^{k-1}(S^{n-1})$.
The isomorphisms 
$H^k\cA_0(\R^n)\longrightarrow \C,\quad k=1,n$
are given by integration along the fiber. 

2. This Theorem generalizes the results of \cite[Sec. 1]{Pay:NRC}
on the characterization of the residue integral and the regularized
integral in terms of the Stokes property. 

3. The proof of the Theorem is based on the Thom isomorphism below.
\end{remark}

\details{
\details{Textbaustein, existence of $\pi_*$ on $\Omega\cA(\R^n)$ is an issue.

integration along the fiber yields an isomorphism $\pi_*: H^k\cA(\R^n)\longrightarrow H^{k-1}(S^{n-1})$.
Thus if $k=1$ and $\go\in\Omega^1\cA(\R^n), d\go=0$ then $\pi_*\go=\gl_\go \cdot 1$ and
$[\go]\mapsto \gl_\go$ is the isomorphism $H^1\cA(\R^n)\longrightarrow \C$.

If $k=n$ then the canonical isomorphism $H^n\cA(\R^n)\longrightarrow\C$ is given by total
integration
\begin{equation}
H^n\cA(\R^n)\ni [\go]\mapsto \int_{S^{n-1}} \pi_*\go.
\end{equation}
}

\commentary{The proof of the Theorem will be based on the Thom isomorphism presented in Sec. ???}

\begin{prop} Let $\cA$ be as in Def. \plref{def-1.1}. Then integration along the fiber
as defined in ???? extends to a homomorphism $\pi_*:\Omega^k\cA(\R^n)\to \Omega^{k-1}(\R^{n-1})$
which commutes with exterior differentiation.
\end{prop}
\begin{proof}
\end{proof}
}

\subsubsection{The Thom isomorphism}

We consider again a Fr{\'e}chet space $\cA$ as in Def. \plref{def-1.1}. Having established
integration along the fiber the Thom isomorphism is proved along the lines of the classical
case of smooth compactly supported forms. The result is as follows:

\begin{theorem}\label{thm:Thom-isom} Let $\cA$ be a \frechet\ algebra as in Def. \plref{def-1.1}. Let $M$ be a compact
oriented manifold of dimension $n$. Furthermore let 
\[\pi_*:\Omega^k\cA_0(\halfline\times M)\longrightarrow
\Omega^{k-1}(\halfline\times M)\]
be integration along the fiber as defined in Section \plref{sec-2.1}.

Then $\pi_*$ induces an isomorphism
\begin{equation}\label{eq:Thom-isom}
  H^k\cA_0(\halfline\times M)\longrightarrow H^{k-1}_{\textup{dR}}(M)
\end{equation}
for all $k\ge 0$ (meaning $H^0\cA_0(\halfline\times M)\simeq \{0\}$.)
\end{theorem}
\details{
\begin{proof}
As in the proof of the Thom isomorphism for compactly supported smooth forms we construct a right inverse
of $\pi_*$ and an appropriate homotopy operator. \TODO{refer to Bott/Tu}
In detail. By Def. \plref{def-1.1} (3) there is a $\phi\in\cA_0$ with $\regint \phi=1$.
We then put
\begin{equation}
\begin{split}
     s_*:\Omega^{k-1}(M)&\longrightarrow \Omega^k\cA_0(\halfline\times M)\\
                     \eta&\mapsto \phi(r) \pi^*\eta\wedge dr.
\end{split}
\end{equation}
Obviously, $s_*$ commutes with $d$ and 
\begin{equation}
\pi_*\circ s_*= \id.
\end{equation}
Next we define a homotopy operator
\begin{equation}\label{eq:homotopy-operator}
K:\Omega^k\cA_0(\halfline\times M)\longrightarrow \Omega^{k-1}\cA_0(\halfline\times M)
\end{equation}
as follows. For a $k$--form $\go$ of the form  \eqref{eq:special-forms}
we put
\begin{equation}\label{eq:homotopy-operator-a}
  K\go:= (-1)^{k-1}\int_0^r\Bigl(f_2(s,p) -\bigl(\regint_0^\infty f_2(\cdot,p)\bigr)\phi(s)\Bigr)ds\; \pi^*\eta_2.
\end{equation}
Note that by construction 
$\reginttext_0^\infty\Bigl(f_2(s,p) -\bigl(\regint_0^\infty f_2(\cdot,p)\bigr)\phi(s)\Bigr)ds=0$ and
thus
$r\mapsto \int_0^r\Bigl(f_2(s,p) -\bigl(\regint_0^\infty f_2(\cdot,p)\bigr)\phi(s)\Bigr)ds$
is in $\cA_0$ by Def. \plref{def-1.1} (4).

Extending $K$ by linearity gives indeed a well--defined homomorphism as claimed in Eq. \eqref{eq:homotopy-operator}.
Exploiting the fact that $\reginttext$ is closed a straightforward calculation yields
\begin{equation}
   dK+Kd=\id-s_*\pi_*.
\end{equation} ya
Thus $K$ is a homotopy operator showing that at the level of cohomology $s_*$ is
the inverse of $\pi_*$. The Theorem is proved.
\end{proof}
}

\bibliography{lesch}

\providecommand{\bysame}{\leavevmode\hbox to3em{\hrulefill}\thinspace}
\providecommand{\MR}{\relax\ifhmode\unskip\space\fi MR }
\providecommand{\MRhref}[2]{%
  \href{http://www.ams.org/mathscinet-getitem?mr=#1}{#2}
}
\providecommand{\href}[2]{#2}
\begin{thebibliography}{KaWa95}

\bibitem[\textsc{BoTu82}]{BotTu:DFA}
\textsc{R.~Bott} and \textsc{L.~W. Tu}, \emph{Differential forms in algebraic
  topology}, Graduate Texts in Mathematics, vol.~82, Springer-Verlag, New York,
  1982. \MR{658304 (83i:57016)}

\bibitem[\textsc{BrLe99}]{BruLes:EIC}
\textsc{J.~Br{\"u}ning} and \textsc{M.~Lesch}, \emph{On the {$\eta$}-invariant
  of certain nonlocal boundary value problems}, Duke Math. J. \textbf{96}
  (1999), no.~2, 425--468. \texttt{arxiv:9609001 [dg-ga,math.DG]}, \MR{1666570
  (99m:58180)}

\bibitem[\textsc{CaZy57}]{CalZyg:SIO}
\textsc{A.-P. Calder{\'o}n} and \textsc{A.~Zygmund}, \emph{Singular integral
  operators and differential equations}, Amer. J. Math. \textbf{79} (1957),
  901--921. \MR{0100768 (20 \#7196)}

\bibitem[\textsc{Con88}]{Con:AFN}
\textsc{A.~Connes}, \emph{The action functional in noncommutative geometry},
  Comm. Math. Phys. \textbf{117} (1988), no.~4, 673--683. \MR{953826
  (91b:58246)}

\bibitem[\textsc{Con94}]{Con:NG}
\textsc{A.~Connes}, \emph{Noncommutative geometry}, Academic Press Inc., San
  Diego, CA, 1994. \MR{1303779 (95j:46063)}

\bibitem[\textsc{CPS03}]{Caretal:SFD}
\textsc{A.~Carey}, \textsc{J.~Phillips}, and \textsc{F.~Sukochev},
  \emph{Spectral flow and {D}ixmier traces}, Adv. Math. \textbf{173} (2003),
  no.~1, 68--113. \texttt{arxiv:0205076 [math.OA]}, \MR{1954456 (2004e:58049)}

\bibitem[\textsc{CRSS07}]{Caretal:DTA}
\textsc{A.~L. Carey}, \textsc{A.~Rennie}, \textsc{A.~Sedaev}, and
  \textsc{F.~Sukochev}, \emph{The {D}ixmier trace and asymptotics of zeta
  functions}, J. Funct. Anal. \textbf{249} (2007), no.~2, 253--283.
  \texttt{arxiv:0611629 [math.OA]}, \MR{2345333}

\bibitem[\textsc{Dix66}]{Dix:ETN}
\textsc{J.~Dixmier}, \emph{Existence de traces non normales}, C. R. Acad. Sci.
  Paris S\'er. A-B \textbf{262} (1966), A1107--A1108. \MR{0196508 (33 \#4695)}

\bibitem[\textsc{Dui96}]{Dui:FIO}
\textsc{J.~J. Duistermaat}, \emph{Fourier integral operators}, Progress in
  Mathematics, vol. 130, Birkh\"auser Boston Inc., Boston, MA, 1996.
  \MR{1362544 (96m:58245)}

\bibitem[\textsc{Gil95}]{Gil:ITH}
\textsc{P.~B. Gilkey}, \emph{Invariance theory, the heat equation, and the
  {A}tiyah-{S}inger index theorem}, second ed., Studies in Advanced
  Mathematics, CRC Press, Boca Raton, FL, 1995. \MR{1396308 (98b:58156)}

\bibitem[\textsc{GrSe95}]{GruSee:WPP}
\textsc{G.~Grubb} and \textsc{R.~T. Seeley}, \emph{Weakly parametric
  pseudodifferential operators and {A}tiyah-{P}atodi-{S}inger boundary
  problems}, Invent. Math. \textbf{121} (1995), no.~3, 481--529. \MR{1353307
  (96k:58216)}

\bibitem[\textsc{GrSj94}]{GriSjo:MAD}
\textsc{A.~Grigis} and \textsc{J.~Sj{\"o}strand}, \emph{Microlocal analysis for
  differential operators}, London Mathematical Society Lecture Note Series,
  vol. 196, Cambridge University Press, Cambridge, 1994, An introduction.
  \MR{1269107 (95d:35009)}

\bibitem[\textsc{Gru05}]{Gru:RAT}
\textsc{G.~Grubb}, \emph{A resolvent approach to traces and zeta {L}aurent
  expansions}, Spectral geometry of manifolds with boundary and decomposition
  of manifolds, Contemp. Math., vol. 366, Amer. Math. Soc., Providence, RI,
  2005, pp.~67--93. \texttt{arxiv:0311081 [math.AP]}, \MR{2114484
  (2006f:58042)}

\bibitem[\textsc{Gui85}]{Gui:NPW}
\textsc{V.~Guillemin}, \emph{A new proof of {W}eyl's formula on the asymptotic
  distribution of eigenvalues}, Adv. in Math. \textbf{55} (1985), no.~2,
  131--160. \MR{772612 (86i:58135)}

\bibitem[\textsc{Had32}]{Had:PCE}
\textsc{J.~Hadamard}, \emph{Le probl{\`e}me de {C}auchy et les \'equations aux
  d\'eriv\'ees partielles lin\'eaires hyperboliques}, Hermann, Paris, 1932.

\bibitem[\textsc{Her61a}]{Her:CVS}
\textsc{J.~Hersch}, \emph{Caract\'erisation variationnelle d'une somme de
  valeurs propres cons\'ecutives; g\'en\'eralisation d'in\'egalit\'es de
  {P}\'olya-{S}chiffer et de {W}eyl}, C. R. Acad. Sci. Paris \textbf{252}
  (1961), 1714--1716. \MR{0126065 (23 \#A3362)}

\bibitem[\textsc{Her61b}]{Her:IVP}
\bysame, \emph{In\'egalit\'es pour des valeurs propres cons\'ecutives de
  syst\`emes vibrants inhomog\`enes allant ``en sens inverse'' de celles de
  {P}\'olya-{S}chiffer et de {W}eyl}, C. R. Acad. Sci. Paris \textbf{252}
  (1961), 2496--2498. \MR{0123908 (23 \#A1229)}

\bibitem[\textsc{H{\"o}r71}]{Hor:FIOI}
\textsc{L.~H{\"o}rmander}, \emph{Fourier integral operators. {I}}, Acta Math.
  \textbf{127} (1971), no.~1-2, 79--183. \MR{0388463 (52 \#9299)}

\bibitem[\textsc{KaRi97}]{KadRin:FTOII}
\textsc{R.~V. Kadison} and \textsc{J.~R. Ringrose}, \emph{Fundamentals of the
  theory of operator algebras. {V}ol. {II}}, Graduate Studies in Mathematics,
  vol.~16, American Mathematical Society, Providence, RI, 1997, Advanced
  theory, Corrected reprint of the 1986 original. \MR{1468230 (98f:46001b)}

\bibitem[\textsc{Kas95}]{Kas:DOG}
\textsc{D.~Kastler}, \emph{The {D}irac operator and gravitation}, Comm. Math.
  Phys. \textbf{166} (1995), no.~3, 633--643. \MR{1312438 (95j:58181)}

\bibitem[\textsc{KaWa95}]{KalWal:GNC}
\textsc{W.~Kalau} and \textsc{M.~Walze}, \emph{Gravity, non-commutative
  geometry and the {W}odzicki residue}, J. Geom. Phys. \textbf{16} (1995),
  no.~4, 327--344. \texttt{arxiv:9312031 [gr-qc]}, \MR{1336738 (96c:58016)}

\bibitem[\textsc{KoNi65}]{KohNir:APD}
\textsc{J.~J. Kohn} and \textsc{L.~Nirenberg}, \emph{An algebra of
  pseudo-differential operators}, Comm. Pure Appl. Math. \textbf{18} (1965),
  269--305. \MR{0176362 (31 \#636)}

\bibitem[\textsc{KoVi94}]{KonVis:DEP}
\textsc{M.~Kontsevich} and \textsc{S.~Vishik}, \emph{Determinants of elliptic
  pseudo-differential operators},  \texttt{arxiv:9404046 [hep-th]}.

\bibitem[\textsc{KoVi95}]{KonVis:GDE}
\bysame, \emph{Geometry of determinants of elliptic operators}, Functional
  analysis on the eve of the 21st century, {V}ol.\ 1 ({N}ew {B}runswick, {NJ},
  1993), Progr. Math., vol. 131, Birkh\"auser Boston, Boston, MA, 1995,
  pp.~173--197. \texttt{arxiv:9406140 [hep-th]}, \MR{1373003 (96m:58264)}

\bibitem[\textsc{LePf00}]{LesPfl:TAP}
\textsc{M.~Lesch} and \textsc{M.~J. Pflaum}, \emph{Traces on algebras of
  parameter dependent pseudodifferential operators and the eta-invariant},
  Trans. Amer. Math. Soc. \textbf{352} (2000), no.~11, 4911--4936. \MR{1661258
  (2001b:58042)}

\bibitem[\textsc{Les97}]{Les:OFT}
\textsc{M.~Lesch}, \emph{Operators of {F}uchs type, conical singularities, and
  asymptotic methods}, Teubner-Texte zur Mathematik [Teubner Texts in
  Mathematics], vol. 136, B. G. Teubner Verlagsgesellschaft mbH, Stuttgart,
  1997. \texttt{arxiv:9607005 [dg-ga, math.DG]}, \MR{1449639 (98d:58174)}

\bibitem[\textsc{Les99}]{Les:NRP}
\bysame, \emph{On the noncommutative residue for pseudodifferential operators
  with log-polyhomogeneous symbols}, Ann. Global Anal. Geom. \textbf{17}
  (1999), no.~2, 151--187. \texttt{arXiv:9708010 [dg-ga,math.DG]}, \MR{1675408
  (2000b:58050)}

\bibitem[\textsc{LMJ09}]{LesMosPfl:RPC}
\textsc{M.~Lesch}, \textsc{H.~Moscovici}, and \textsc{P.~M. J.}, \emph{Relative
  pairing in cyclic cohomology and divisor flows}, J. K-Theory \textbf{3}
  (2009), 359--407. \texttt{arXiv:0603500 [math.KT]}

\bibitem[\textsc{Mel95}]{Mel:EIF}
\textsc{R.~B. Melrose}, \emph{The eta invariant and families of
  pseudodifferential operators}, Math. Res. Lett. \textbf{2} (1995), no.~5,
  541--561. \MR{1359962 (96h:58169)}

\bibitem[\textsc{MiPl49}]{MinPle:SPE}
\textsc{S.~Minakshisundaram} and \textsc{{\AA}.~Pleijel}, \emph{Some properties
  of the eigenfunctions of the {L}aplace-operator on {R}iemannian manifolds},
  Canadian J. Math. \textbf{1} (1949), 242--256. \MR{0031145 (11,108b)}

\bibitem[\textsc{MMP05}]{Manetal:SFC}
\textsc{D.~Manchon}, \textsc{Y.~Maeda}, and \textsc{S.~Paycha}, \emph{Stokes'
  formulae on classical symbol valued forms and applications},
  \texttt{arxiv:0510454 [math.DG]}.

\bibitem[\textsc{MuvN36}]{MurNeu:ROI}
\textsc{F.~J. Murray} and \textsc{J.~von Neumann}, \emph{On rings of
  operators}, Ann. of Math. (2) \textbf{37} (1936), no.~1, 116--229.
  \MR{MR1503275}

\bibitem[\textsc{MuvN37}]{MurNeu:ROII}
\bysame, \emph{On rings of operators. {II}}, Trans. Amer. Math. Soc.
  \textbf{41} (1937), no.~2, 208--248. \MR{1501899}

\bibitem[\textsc{MuvN43}]{MurNeu:ROIV}
\bysame, \emph{On rings of operators. {IV}}, Ann. of Math. (2) \textbf{44}
  (1943), 716--808. \MR{0009096 (5,101a)}

\bibitem[\textsc{Pay05}]{Pay:NRC}
\textsc{S.~Paycha}, \emph{The noncommutative residue and canonical trace in the
  light of {S}tokes' and continuity properties},  \texttt{arXiv:0706.2552
  [math.OA]}.

\bibitem[\textsc{Ped89}]{Ped:AN}
\textsc{G.~K. Pedersen}, \emph{Analysis now}, Graduate Texts in Mathematics,
  vol. 118, Springer-Verlag, New York, 1989. \MR{971256 (90f:46001)}

\bibitem[\textsc{See59}]{See:SIC}
\textsc{R.~T. Seeley}, \emph{Singular integrals on compact manifolds}, Amer. J.
  Math. \textbf{81} (1959), 658--690. \MR{0110022 (22 \#905)}

\bibitem[\textsc{See65}]{See:IDO}
\bysame, \emph{Integro-differential operators on vector bundles}, Trans. Amer.
  Math. Soc. \textbf{117} (1965), 167--204. \MR{0173174 (30 \#3387)}

\bibitem[\textsc{See67}]{See:CPE}
\bysame, \emph{Complex powers of an elliptic operator}, Singular {I}ntegrals
  ({P}roc. {S}ympos. {P}ure {M}ath., {C}hicago, {I}ll., 1966), Amer. Math.
  Soc., Providence, R.I., 1967, pp.~288--307. \MR{0237943 (38 \#6220)}

\bibitem[\textsc{Shu01}]{Shu:POST}
\textsc{M.~A. Shubin}, \emph{Pseudodifferential operators and spectral theory},
  second ed., Springer-Verlag, Berlin, 2001, Translated from the 1978 Russian
  original by Stig I. Andersson. \MR{1852334 (2002d:47073)}

\bibitem[\textsc{vN40}]{Neu:ROIII}
\textsc{J.~v.~Neumann}, \emph{On rings of operators. {III}}, Ann. of Math. (2)
  \textbf{41} (1940), 94--161. \MR{0000898 (1,146g)}

\bibitem[\textsc{Wod84}]{Wod:LIS}
\textsc{M.~Wodzicki}, \emph{Local invariants of spectral asymmetry}, Invent.
  Math. \textbf{75} (1984), no.~1, 143--177. \MR{728144 (85g:58089)}

\bibitem[\textsc{Wod87}]{Wod:NRF}
\textsc{M.~Wodzicki}, \emph{Noncommutative residue. {I}. {F}undamentals},
  {$K$}-theory, arithmetic and geometry ({M}oscow, 1984--1986), Lecture Notes
  in Math., vol. 1289, Springer, Berlin, 1987, pp.~320--399. \MR{923140
  (90a:58175)}

\end{thebibliography}
\bibliographystyle{amsalpha-lmp}
\end{document}